\documentclass[12pt, reqno]{amsart}
\usepackage{amsmath,amssymb,amsfonts,comment,verbatim,amsthm}
\numberwithin{equation}{section}
\allowdisplaybreaks[2]

\newtheorem{proposition}{Proposition}[section]
\newtheorem{theorem}{Theorem}[section]
\newtheorem{lemma}[proposition]{Lemma}
\newtheorem{claim}[proposition]{Claim}

\newtheorem{corollary}{Corollary}[section]
\theoremstyle{remark}

\begin{document}

\title[Well-posedness of the Stochastic KdV-Burgers Equation]{Local and Global Well-posedness of the Stochastic KdV-Burgers Equation}
\author{Geordie Richards}
\address{Department of Mathematics, University of Toronto, Toronto, Ontario, Canada}
\email{grichard@math.toronto.edu}
\subjclass[2000]{35Q53, 35R60, 60H15}
\keywords{well-posedness, stochastic PDEs, white noise invariance}
\thanks{This is part of the author's forthcoming Ph.D. thesis at the University of Toronto advised by J. Colliander and T. Oh.  The author also thanks J. Quastel for suggestions related to this work.}
\maketitle
\begin{abstract}

The stochastic PDE known as the Kardar-Parisi-Zhang equation (KPZ) has been proposed as a model for a randomly growing interface.  This equation can be reformulated as a stochastic Burgers equation.
We study a stochastic KdV-Burgers equation as a toy model for this stochastic Burgers equation.
Both of these equations formally preserve spatial white noise.  We are interested in rigorously proving the invariance of white noise for the stochastic KdV-Burgers equation.  This paper establishes a result in this direction.  After smoothing the additive noise (by less than one spatial derivative), we establish (almost sure) local well-posedness of the stochastic KdV-Burgers equation with white noise as initial data.  We also prove a global well-posedness result under an additional smoothing of the noise.
\end{abstract}

\tableofcontents

\section{Introduction}
\label{Sec:intro}
In this paper we study the stochastic Korteweg-de Vries (KdV) Burgers equation
\begin{align}
\left\{
\begin{array}{ll}du = (u_{xx} - u_{xxx} - (u^{2})_{x})dt + \phi\partial_{x} dW  , \ \
t\geq 0, x\in \mathbb{T}
\\
u\big|_{t=0} = u_{0},
\end{array} \right.
\label{Eqn:SDKDV}
\end{align}
where $\phi$ is a bounded operator on $L^{2}(\mathbb{T})$, and $W(t,x)$ is a cylindrical white noise of the form
\begin{align}
W(t,x) = \sum_{n\neq 0} B_{n}(t)e^{inx}.
\label{Eqn:STwhite}
\end{align}

\noindent Here $(B_{n}(t))_{n\in \mathbb{N}}$ is a family of standard complex-valued Brownian motions mutually
independent in a filtered probability space $(\Omega,\mathcal{F},(\mathcal{F}_{t})_{t\geq0},\mathbb{P})$,
and $B_{-n}=\overline{B_{n}}$, since we are interested in real-valued noise.

We consider \eqref{Eqn:SDKDV} as a toy model for the stochastic Burgers equation
\begin{align}
\left\{
\begin{array}{ll}du = (u_{xx} - (u^{2})_{x})dt + \partial_{x} dW  , \ \
t\geq 0, x\in \mathbb{T},
\\
u\big|_{t=0} = u_{0}.
\end{array} \right.
\label{Eqn:SB}
\end{align}
\noindent The equation \eqref{Eqn:SB} is a reformulation of the Kardar-Parisi-Zhang equation (KPZ).  That is, letting $u=\partial_{x}h$, $u$ satisfies \eqref{Eqn:SB} if and only if $h$ satisfies KPZ, given by
\begin{align}
dh = (h_{xx} + \frac{1}{2}(h_{x})^{2})dt + dW, \ \
t\geq 0, x\in \mathbb{T}.
\label{Eqn:KPZ}
\end{align}
The equation \eqref{Eqn:KPZ} was introduced \cite{KPZ} to model the fluctuations (over long scales) of a growing interface.  For example, $h(t,x)$ could describe the height of an interface between regions of opposite polarity inside a ferromagnet subject to an external magnetic field.  Mathematical interest in \eqref{Eqn:KPZ} is motivated by evidence that second order dependence on the derivative $\partial_{x}h$ (over long scales) is universal - that is, independent of the microscopic dynamics - within a certain class of growth models \cite{KS}.  This is verified mathematically in Bertini-Giacomin \cite{BG} for a specific growth model; they obtain \eqref{Eqn:KPZ} as the limit equation of a suitable particle system.

\subsection{Background}
By local well-posedness (LWP) of a stochastic PDE we mean pathwise LWP almost surely.
That is, for almost every fixed $\omega\in\Omega$, the corresponding PDE is LWP.  Similarly, global well-posedness (GWP) of a stochastic PDE will be defined as pathwise GWP almost surely.

Well-posedness of \eqref{Eqn:SB} (and \eqref{Eqn:KPZ}) is an open problem.
Along this direction, previous studies have considered modified equations.  One strategy is to solve a stochastic heat equation which is \textit{formally} equivalent to \eqref{Eqn:KPZ} under the Hopf-Cole transformation, as in Krug-Spohn \cite{KS}.  Another strategy is to modify the stochastic Burgers equation \eqref{Eqn:SB} directly.  For example, Da Prato-Debussche-Temam \cite{DaDe} considered \eqref{Eqn:SB} without the spatial derivative $\partial_{x}$ applied the noise.  That is, they smoothed the additive noise by one derivative in space.  Then, they established local (and global) well-posedness using a fixed point theorem.  Our approach is closer to this strategy, but we modify the linear part of \eqref{Eqn:SB} instead.  By including $-u_{xxx}$ in \eqref{Eqn:SB}, we obtain \eqref{Eqn:SDKDV} (with $\phi =\text{Id}$).

\medskip

With $\phi=\text{Id}$, the equations \eqref{Eqn:SDKDV} and \eqref{Eqn:SB} share a physically significant property; both of these equations \textit{formally} preserve mean zero spatial white noise.  Mean zero spatial white noise is the unique probability measure $\mu$ on the space of mean zero distributions on $\mathbb{T}$ satisfying
\begin{align}
\int e^{i\langle \lambda,u\rangle}d\mu(u)=
e^{-\frac{1}{2}\|\lambda\|_{2}^{2}},
\label{Eqn:WNonHs}
\end{align}
for any mean zero smooth function $\lambda$
on $\mathbb{T}$.  Here $\|\cdot\|_{2}^{2}=\langle\cdot,\cdot\rangle$ denotes the $L^{2}(\mathbb{T},dx)$ norm defined through the $\mathcal{D}$-$\mathcal{D}'$ duality.
From now on, we assume that the spatial mean is always zero, and omit the prefix ``mean zero''.
Let $\{e_{n}\}_{n\in \mathbb{N}}$ be an orthonormal basis of smooth functions in $L^{2}(\mathbb{T})$.  Then white noise is represented as $u = \sum_{n=1}^{\infty}g_{n}e_{n}$, where $\{g_{n}\}_{n\in\mathbb{N}}$ is a sequence of independent Gaussian random variables with mean zero and variance 1.  In particular, white noise is supported in $H^{s}(\mathbb{T})$ for any $s<-\frac{1}{2}$.

It is important for us to clarify that our analysis involves two distinct types of noise.  In the equation \eqref{Eqn:SDKDV}, the term $W(t,x)$ represents \textit{space-time} white noise, as defined in \eqref{Eqn:STwhite}.
Separately, we have just defined \textit{spatial} white noise, a probability measure supported on $H^{s}(\mathbb{T})$ for any $s<-\frac{1}{2}$, represented as $u = \sum_{n=1}^{\infty}g_{n}e_{n}$.  While space-time white noise is a term in our stochastic PDE, we will consider spatial white noise as initial data.

Recall that if $f : X_{1} \longrightarrow X_{2}$ is a measurable map between metric spaces and $\mu$ is a probability measure on $(X_{1}, B(X_{1}))$, then the pushforward $f^{*}\mu$ is the measure on $X_{2}$ given by $f^{*}\mu(A) = \mu(\{x : f(x) \in A\})$ for any Borel set $A \in B(X_{2})$.  For $s<-\frac{1}{2}$, the flow map $S_{t}:H^{s}(\mathbb{T})\rightarrow
H^{s}(\mathbb{T})$ of a stochastic PDE preserves white noise (or equivalently, white noise is invariant under the flow) if $S_{t}^{*}\mu = \mu$ for each time $t\geq 0$.

In \cite{BG}, the invariance of white noise is rigorously proven for a stochastic process which is \textit{formally} a solution to \eqref{Eqn:SB} through the Hopf-Cole transformation; this is how we interpret the formal invariance of white noise for \eqref{Eqn:SB}.  In the appendix, we present a formal proof of white noise invariance for \eqref{Eqn:SDKDV}.  This is based on decomposing \eqref{Eqn:SDKDV} into the KdV equation, plus a rescaled Ornstein-Uhlenbeck process at each spatial frequency; each of these evolutions individually preserve white noise.  Indeed for the KdV equation, the invariance of white noise has been proven rigorously; see Quastel-Valko\cite{QV}, Oh-Quastel-Valko\cite{OQV} and Oh \cite{OH}.

The common feature of (formal) white noise invariance motivates our study of \eqref{Eqn:SDKDV} as a toy model for \eqref{Eqn:SB}.  We are interested in rigorously proving the invariance of white noise for \eqref{Eqn:SDKDV}.  Our main result is a step in this direction (see Theorem \ref{Thm:Contraction} and Corollary \ref{Cor:whitenoise}).


\subsection{Fixed point approach}

We will solve \eqref{Eqn:SDKDV} in the mild formulation
\begin{align}
u(t)= S(t)u_{0} - \frac{1}{2}\int_{0}^{t}S(t-t')\partial_{x}(u^2(t'))dt'
+ \int_{0}^{t}S(t-t')\phi\partial_{x}dW(t').
\label{Eqn:SDKDV-Duh-1}
\end{align}

\noindent  In the equation \eqref{Eqn:SDKDV-Duh-1}, $S(t)u_{0}$ denotes the solution to the linear KdV-Burgers equation with initial data $u_{0}$ evaluated at time $t$, suitably extended to all $t\in\mathbb{R}$.  That is, $\widehat{S(t)u_{0}}(n)=e^{-n^{2}|t|+in^{3}t}\widehat{u_{0}}(n)$ for any $t\in\mathbb{R}$, $n\in\mathbb{Z}$.  To simplify notation, we let
$$ \Phi(t):= \int_{0}^{t}S(t-t')\phi\partial_{x}dW(t').$$
The function $\Phi(t)=\Phi_{\omega}(t)$ is referred to as the stochastic convolution.

Our analysis will take place in the $X^{s,b}$ space of functions of space-time adapted to the (linear) KdV-Burgers equation.  As in \cite{MR}, for $s,b \in \mathbb{R}$, the $X^{s,b}$ space is a weighted Sobolev space whose norm is given by
\begin{align*}
\|u\|_{X^{s,b}} = \Bigg(\int_{\tau\in\mathbb{R}}\sum_{n\in \mathbb{Z}}\langle n\rangle^{2s}\langle i(\tau-n^{3})+n^{2}\rangle^{2b}|\tilde{u}(n,\tau)|^{2}d\tau\Bigg)^{1/2}.
\end{align*}

\noindent Recall that $\langle \cdot \rangle := (1+|\cdot|^{2})^{\frac{1}{2}}$.  The time restricted $X^{s,b}_{T}$ space is defined with the norm
\begin{align*}
\|u\|_{X^{s,b}_{T}} = \inf\big\{\|v\|_{X^{s,b}}: v\in X^{s,b} \text{ and } v(t)\equiv u(t) \text{ for} \ t\in[0,T] \big\}.
\end{align*}
We define the local-in-time version $X^{s,b}_{I}= X^{s,b}_{[a, b]}$ on an interval $I = [a, b]$ in an analogous way.

\noindent The $X^{s,b}$ spaces were used by Bourgain \cite{B1} to study the NLS and KdV equations.  Bourgain established LWP of KdV in $L^{2}(\mathbb{T})$ using a fixed point theorem in the $X^{s,b}$ space (adapted to KdV).  He automatically obtained GWP by invariance of the $L^{2}(\mathbb{T})$ norm.  This was improved to LWP in $H^{-\frac{1}{2}+}(\mathbb{T})$ by Kenig-Ponce-Vega \cite{KPV2}, and the corresponding global result was obtained by Colliander-Keel-Staffilani-Takaoka-Tao \cite{CKSTT} using the I-method (they also obtained GWP at the endpoint $s=-\frac{1}{2}$).  These (local-in-time) results are based primarily on bilinear estimates in (sometimes modified) $X^{s,b}$ spaces.  For example, in \cite{KPV2}, the estimate
\begin{align}
\|\partial_{x}(uv)\|_{X^{s,b-1}} \leq \|u\|_{X^{s,b}}\|v\|_{X^{s,b}}
\label{Eqn:KPV-bilin}
\end{align}
is established for $s\geq -\frac{1}{2}$, $b = \frac{1}{2}$.  This estimate is then used to prove LWP of KdV with a fixed point method.


It is also shown in \cite{KPV2}
that \eqref{Eqn:KPV-bilin} fails if  $s<-\frac{1}{2}$, or $b<\frac{1}{2}$.  That is, the bilinear $X^{s,b}$ estimate used for well-posedness of KdV requires spatial regularity $s\geq -\frac{1}{2}$ and temporal regularity $b\geq \frac{1}{2}$.  In a related result, it is shown in \cite{B2} that the data to solution map $u_{0}\in H^{s}(\mathbb{T})\longmapsto u(t)\in H^{s}(\mathbb{T})$ for KdV is not $C^{3}$ for $s<-\frac{1}{2}$.  This indicates that there is no hope in using a fixed point theorem to establish well-posedness of KdV in $H^{s}(\mathbb{T})$ for $s<-\frac{1}{2}$; fixed point theorems guarantee analyticity of the data to solution map.  Kappeler-Topalov \cite{KT} established GWP of KdV in $H^{-1}(\mathbb{T})$ using the inverse scattering technique, but the data to solution map is merely continuous.  In a separate result, Dix \cite{Dix} established ill-posedness of the deterministic Burgers equation in $H^{s}(\mathbb{R})$ for $s<-\frac{1}{2}$ due to a lack of uniqueness.  We observe that, for both the deterministic KdV and the deterministic Burgers equation, $s=-\frac{1}{2}$ is the optimal regularity for well-posedness in $H^{s}$ via the fixed point method.

The barrier of $s=-\frac{1}{2}$ is not present in the analysis of the deterministic KdV-Burgers equation (\eqref{Eqn:SDKDV} with $\phi=0$).  See, for example, Molinet-Ribaud \cite{MR} and Molinet-Vento \cite{MV}.  In particular, LWP of the KdV-Burgers equation in $H^{-1}(\mathbb{T})$ is established in \cite{MV} using a fixed point method; the data to solution map is analytic.  Moreover, in \cite{MR}, an estimate of the form \eqref{Eqn:KPV-bilin}
is established for any $s>-1$, $b= \frac{1}{2}$.  In this way, the combination of dispersion and dissipation in the KdV-Burgers equation allows for an improved bilinear estimate of the form \eqref{Eqn:KPV-bilin} (in particular, with spatial regularity $s<-\frac{1}{2}$), which yields superior LWP results.


\subsection{Regularity of the noise}

Let us now discuss function spaces which capture the regularity of space-time white noise.  Spaces of this type were considered in De Bouard-Debussche-Tsutsumi \cite{DDT},
and Oh \cite{OH}, to study the stochastic KdV (with additive noise), given by
\begin{align}
du = (-u_{xxx} - (u^{2})_{x})dt + \phi dW  , \quad  t\geq 0, x\in \mathbb{T}.
\label{Eqn:SKDV}
\end{align}

\noindent The argument in \cite{DDT} is based on the result of Roynette \cite{Ro} on the endpoint regularity of Brownian motion in a Besov space.  They proved a variant of the bilinear estimate \eqref{Eqn:KPV-bilin} adapted to their Besov space setting, establishing LWP of \eqref{Eqn:SKDV} using the fixed point method.  However, the modified bilinear estimate required a slight regularization of the noise in space via the bounded operator $\phi$, so that the smoothed noise has spatial regularity $s > -\frac{1}{2}$.  In particular, they could not treat the case of space-time white noise, as written in \eqref{Eqn:SKDV} (i.e. $\phi = Id$).

These developments are clarified in \cite{OH} with two observations.  The first observation is that a certain modified Besov space, and the corresponding $X^{s,b}$-type space, capture the regularity of spatial and space-time white noise, respectively, for $s<-\frac{1}{2}$ and $b<\frac{1}{2}$.  The second is that apriori estimates on the \textit{second iteration} of the mild formulation of \eqref{Eqn:SKDV} (in these spaces) are sufficient for LWP of \eqref{Eqn:SKDV}.  The condition $b<\frac{1}{2}$ is dictated by the space-time white noise, highlighting the challenge behind proving well-posedness for \eqref{Eqn:SKDV}; recall that the bilinear $X^{s,b}$ estimate \eqref{Eqn:KPV-bilin} for KdV requires $b\geq \frac{1}{2}$.

Notice that the additive noise in \eqref{Eqn:SKDV} lacks the spatial derivative $\partial_{x}$ appearing in \eqref{Eqn:SB}.  Indeed, an endpoint result (\cite{OH}, in terms of roughness of the noise using a fixed point method) for the stochastic KdV incorporates noise which is a full derivative smoother than the noise in \eqref{Eqn:SB}.  In a similar way, the noise considered in \cite{DaDe} is a full derivative smoother than the noise appearing in \eqref{Eqn:SB}.

\subsection{Results}

The main result of this paper is LWP of \eqref{Eqn:SDKDV} with a rougher additive noise than in \cite{DaDe,OH} (Theorem \ref{Thm:Contraction}).  We exploit the combination of dispersion and dissipation in (the linear part of) \eqref{Eqn:SDKDV} to relax the conditions placed on the noise in previous analyses of the stochastic Burgers equation \cite{DaDe} (dissipation only) and the stochastic KdV \cite{OH} (dispersion only).  This result holds (almost surely) with spatial white noise as initial data; it is a first step towards a rigorous proof of the white noise invariance for \eqref{Eqn:SDKDV}.  Indeed, to prove the white noise invariance, one must first establish well-defined dynamics (almost surely) for \eqref{Eqn:SDKDV} in the support of white noise.  The combination of dispersion and dissipation in the KdV-Burgers propagator helps our analysis in the following way: we can establish a bilinear estimate of the form \eqref{Eqn:KPV-bilin} (see Proposition \ref{Prop:bilinear} below) in the $X^{s,b}_{T}$ space adapted to the KdV-Burgers equation, with $s<-\frac{1}{2}$ and $b<\frac{1}{2}$ (recall that for the $X^{s,b}_{T}$ space adapted to the KdV equation, this requires $s\geq -\frac{1}{2}$ and $b\geq \frac{1}{2}$).  Taking $s<-\frac{1}{2}$, we can treat white noise as initial data.  With $b<\frac{1}{2}$, the $X^{s,b}_{T}$ space captures the regularity of the space-time white noise.  That is, the stochastic convolution $\Phi(t)$ is almost surely an element of $X^{s,b}_{T}$, under appropriate conditions on $\phi$ (see Proposition \ref{Prop:EstofStochConv}).  Finally, the dissipative semigroup has a smoothing effect in space, and the conditions imposed on $\phi$ by Proposition \ref{Prop:EstofStochConv} allow for a rougher space-time noise than in \cite{DaDe,OH}.  Specifically, we can relax the smoothing of the noise to $s+2b+\frac{1}{2}= \frac{13}{16}+<1$ spatial derivatives.

\medskip

We pause to introduce some notation.  Let us discuss the technical conditions imposed on the smoothing operator $\phi$.  For our purposes, $\phi$ is assumed to be a Hilbert-Schmidt operator from $L^{2}(\mathbb{T})$ to $H^{s}(\mathbb{T})$ (written $\phi\in HS(L^{2};H^{s})$) for some $s\in\mathbb{R}$.  The space $HS(L^{2};H^{s})$ is endowed with its natural norm
\begin{align*}
\|\phi\|_{HS(L^{2};H^{s})} = \Bigg(\sum_{n\in \mathbb{N}}\|\phi e_{n}\|_{H^{s}}^{2}\Bigg)^{1/2},
\end{align*}
where $(e_{n})_{n\in\mathbb{N}}$ is any complete orthonormal system in $L^{2}(\mathbb{T})$.  We will further assume that $\phi$ is a convolution operator.  That is, for any function $f$ on $\mathbb{T}$, and any $n\in \mathbb{Z}$, we have
\begin{align}
\widehat{(\phi f)}(n) = \phi_{n}\widehat{f}(n),
\label{Eqn:phi-diagonal}
\end{align}
for some $\phi_{n}\in\mathbb{R}$.  For example, in \cite{DaDe}, they consider $\phi=\partial_{x}^{-1}$, which corresponds to $\phi_{n}=\frac{1}{n}$, for each $n$.  Using the standard Fourier basis for $L^{2}(\mathbb{T})$, we find that for convolution operators $\phi$, we have $\|\phi\|_{HS(L^{2};H^{s})} = (\sum_{n}\langle n\rangle^{2s}|\phi_{n}|^{2})^{1/2} = \|\phi\|_{H^{s}}$ (with a slight abuse of notation).  The point to take away is that by placing $\phi \in HS(L^{2},H^{s})$, we are smoothing the additive noise in \eqref{Eqn:SDKDV} by $s+\frac{1}{2}$ spatial derivatives.

Other notation we will use includes $a\wedge b : = \min(a,b)$ and $a\vee b : = \max(a,b)$.  Also, we will use $\hat{u}$ to denote the Fourier transform in space, and $\tilde{u}$ to denote the space-time Fourier transform.

\subsection{Statement of Theorems}

Here is the first result we obtain.

\begin{theorem}[Local well-posedness] Given $0<\varepsilon<\frac{1}{16}$, let $s\geq -\frac{1}{2}-\varepsilon$.  Suppose $\phi\in HS(L^{2};H^{s+1-2\varepsilon})$ of the form \eqref{Eqn:phi-diagonal}.  Then \eqref{Eqn:SDKDV} is LWP in $H^{s}(\mathbb{T})$ for mean zero data.  That is, if $u_{0}\in H^{s}(\mathbb{T})$ has mean zero, there exists a stopping time $T_{\omega}>0$ and a unique process $u\in C([0,T_{\omega}];H^{s}(\mathbb{T}))$ satisfying \eqref{Eqn:SDKDV} on $[0,T_{\omega}]$ almost surely.
\label{Thm:Contraction}
\end{theorem}

With $\varepsilon = \frac{1}{16}-$ and $s=-\frac{1}{2}-\varepsilon<-\frac{1}{2}$, we have $s+1-2\varepsilon = \frac{5}{16}+$.  Then
$\phi = \partial_{x}^{-(\frac{13}{16}+)}$ satisfies
$\phi \in HS(L^{2},H^{\frac{5}{16}+})$, as required to apply Theorem \ref{Thm:Contraction}.
With this choice of $\phi$, we have smoothed the noise in \eqref{Eqn:SDKDV} by $\frac{13}{16}+$ spatial derivatives.  In contrast, the space-time noise in \cite{DaDe} is smoothed out by a full derivative in space (by taking $\phi = \partial_{x}^{-1}$).  Recall that spatial white noise is represented as $u=\sum_{n}g_{n}e_{n} \in H^{s}(\mathbb{T})$
almost surely for $s<-\frac{1}{2}$.  It follows that Theorem \ref{Thm:Contraction} applies almost surely with initial data given by spatial white noise.  We have arrived at a consequence of Theorem \ref{Thm:Contraction} worth mentioning.

\begin{corollary}
Let $\phi = \partial_{x}^{-(\frac{13}{16}+)}$, then \eqref{Eqn:SDKDV} is almost surely LWP with spatial white noise as initial data.
\label{Cor:whitenoise}
\end{corollary}



\vspace{0.1in}


\noindent  After proving Theorem \ref{Thm:Contraction}, we establish global well-posedness of \eqref{Eqn:SDKDV} with a smoothed noise.

\begin{theorem}[Global well-posedness]  Let $\phi\in HS (L^{2};H^{1})$.
Then \eqref{Eqn:SDKDV} is GWP in $L^{2}(\mathbb{T})$ for mean zero data.  That is, if $u_{0}\in L^{2}(\mathbb{T})$ has mean zero, then for any $T>0$ there is a unique process $u\in C([0,T];L^{2}(\mathbb{T}))$ satisfying \eqref{Eqn:SDKDV} on $[0,T]$ almost surely.
\label{Thm:GWP}
\end{theorem}


This paper is organized as follows.  In Section \ref{Sec:LWP}
we present local well-posedness (Theorem \ref{Thm:Contraction}).  In Section \ref{Sec:GWP} we present global well-posedness (Theorem \ref{Thm:GWP}).
In the appendix, Section \ref{Sec:app}, we include a formal proof of white noise invariance for \eqref{Eqn:SDKDV}, followed by the proofs of Propositions (\ref{Prop:bilinear} and \ref{Prop:finiteGWP}), which are the longest and most technical proofs appearing in the paper.

\section{Local well-posedness}
\label{Sec:LWP}

In this Section we prove Theorem \ref{Thm:Contraction}.
For each fixed $\omega\in\Omega$, we will prove LWP using the Fourier restriction norm method.
More precisely, we will prove that the solution of \eqref{Eqn:SDKDV-Duh-1} is almost surely the unique fixed point of a contraction on $S(t)u_{0} + \Phi(t) + B$, where $B$ is the unit ball in the space $X^{s,b}_{T_{\omega}}$ of space-time functions (adapted to the KdV-Burgers equation), for a stopping time $T_{\omega}>0$, and suitable $s,b \in \mathbb{R}$.

\subsection{Local estimates}

The proof of Theorem \ref{Thm:Contraction} will require five key Propositions, and a lemma, concerning the $X^{s,b}_{T}$ spaces.  In this subsection we present these Propositions (and lemmata).

\begin{proposition}[Molinet-Ribaud \cite{MR}, Linear estimate]  For any $s\in\mathbb{R}$,
\begin{align*}
\|S(t)\phi\|_{X^{s,\frac{1}{2}}_{T}} \leq \|\phi\|_{H^{s}(\mathbb{T})},
\end{align*}
for all $\phi \in H^{s}(\mathbb{T})$.
\label{Prop:Linear}
\end{proposition}

\begin{proposition}[Molinet-Ribaud \cite{MR}, Non-homogeneous linear estimate]  For any $s\in\mathbb{R}$, $\gamma>0$,
\begin{align*}
\bigg\|\int_{0}^{t}S(t-t')v(t')dt'\bigg\|_{X^{s,\frac{1}{2}}_{T}}
\leq \|v\|_{X^{s,-\frac{1}{2}+\delta}_{T}}
\end{align*}
for all $v \in X^{s,-\frac{1}{2}+\delta}_{T}$.
\label{Prop:Nonhom-Linear}
\end{proposition}

\noindent Propositions \ref{Prop:Linear} and \ref{Prop:Nonhom-Linear} are both established in \cite[Propositions 2.1 and 2.3]{MR}.  We will prove the following bilinear estimate.

\begin{proposition}[Bilinear estimate]  For any $0<\varepsilon<\frac{1}{16}$, if $s\geq -\frac{1}{2}-\varepsilon$, $0<\gamma\leq \varepsilon$, we have the following estimate
\begin{align}
\|(uv)_{x}\|_{X^{s,-\frac{1}{2}+\gamma}_{T}}
\leq \|u\|_{X^{s,\frac{1}{2}-\varepsilon}_{T}}\|v\|_{X^{s,\frac{1}{2}-\varepsilon}_{T}}
\label{Eqn:bilinear}
\end{align}
for all $u, v \in X^{s,\frac{1}{2}-\varepsilon}_{T}$.
\label{Prop:bilinear}
\end{proposition}

\noindent The novelty of Proposition \ref{Prop:bilinear} is the use of temporal regularity $b=\frac{1}{2}-\varepsilon <\frac{1}{2}$ (on the right-hand side).  Recall that for the $X^{s,b}$ space adapted to the KdV equation, this estimate requires $b\geq\frac{1}{2}$.  This modification relies on the combination of dispersion and dissipation in the KdV-Burgers propagator.  Specifically, we will benefit from the algebraic identity identified of Bourgain for the KdV equation, $\max(\tau-n^{3},\tau_{1}-n_{1}^{3},\tau_{2}-n_{2}^{3})\geq n n_{1}n_{2}$ (dispersion), but we will also exploit the lower bound on the KdV-Burgers weight $|i(\tau_{j}-n_{j}^{3})+n_{j}^{2}|\geq n_{j}^{2}$ (dissipation) nearby the dispersive hypersurface $\tau_{j}= n_{j}^{3}$.  The proof of Proposition \ref{Prop:bilinear} is included in the appendix.

Our fixed point argument will take place on a unit ball in $X^{s,b}_{T}$ centered at $S(t)u_{0}+\Phi(t)$, and thus we will require that $S(t)u_{0}+\Phi(t)\in X^{s,b}_{T}$, almost surely.  The linear evolution $S(t)u_{0}$ lives in $X^{s,b}_{T}$ according to Proposition \ref{Prop:Linear}.  To show that the stochastic convolution $\Phi(t)$ is almost surely an element of $X^{s,b}_{T}$, we compute

\begin{proposition}[Stochastic convolution estimate] Let $\phi\in HS(L^{2};H^{s+2b})$ of the form $(\ref{Eqn:phi-diagonal})$.
Given $b<\frac{1}{2}$ and $T>0$, we have
\begin{align*}
\mathbb{E}\Big(\|\Phi\|^{2}
_{X^{s,b}_{T}}\Big) \lesssim T\|\phi\|_{H^{s+2b}}^{2}.
\end{align*}
\label{Prop:EstofStochConv}
\end{proposition}

\noindent According to Proposition \ref{Prop:EstofStochConv}, if $\phi \in HS(L^{s},H^{s+2b})$, then $\Phi(t)\in X^{s,b}_{T}$ almost surely.  Recall that placing $\phi \in HS(L^{s},H^{s+2b})$ corresponds to smoothing the additive noise in \eqref{Eqn:SDKDV} by $s+2b+\frac{1}{2}$ spatial derivatives.  The primary observation of this paper is that, with the combination of dispersion and dissipation, we can obtain the bilinear estimate needed for well-posedness of the KdV-Burgers equation (Proposition \ref{Prop:bilinear}) in a function space ($X^{s,b}_{T_{\omega}}$) with the regularity of white noise in space ($s<-\frac{1}{2}$) and low regularity in time ($b<\frac{1}{2}$).  Then, using Proposition \ref{Prop:EstofStochConv}, we can prove LWP of \eqref{Eqn:SDKDV} with $\phi \in HS(L^{2},H^{s+2b})$,
and we have smoothed the additive noise by $s+2b +\frac{1}{2}<1$
spatial derivatives.  The proof of Proposition \ref{Prop:EstofStochConv} is found in the next subsection.

Using $b<\frac{1}{2}$ requires that we justify continuity of the solution with a separate argument (it is not automatic from the fixed point method).  To establish continuity of the nonlinear part of the solution, we appeal to the following Proposition of \cite{MR}.

\begin{proposition}[Molinet-Ribaud \cite{MR}, Continuity of the Duhamel map] Let  $s\in \mathbb{R}$ and $\gamma>0$. For all $f\in X^{s,-\frac{1}{2}+\gamma}$,
\begin{align}
t\longmapsto \int_{0}^{t}S(t-t')f(t')dt'
\in  C(\mathbb{R}_{+},H^{s+2\gamma}).
\end{align}
Moreover, if $\{f_{n}\}$ is a sequence with $f_{n}\rightarrow 0$ in
$X^{s,\frac{1}{2}+\gamma}$ as $n\rightarrow \infty$, then
\begin{align}
\bigg\|\int_{0}^{t}S(t-t')f_{n}(t')dt'\bigg\|_{L^{\infty}
(\mathbb{R}_{+},H^{s+2\gamma})}  \longrightarrow 0,
\end{align}
as $n\rightarrow \infty$.
\label{Prop:ContofNonlin}
\end{proposition}

\noindent For continuity of the stochastic convolution, we use a probabilistic argument to establish the following Proposition.

\begin{proposition}[Continuity of the stochastic convolution] Let $\phi\in HS(L^{2},H^{s+2\alpha})$ of the form \eqref{Eqn:phi-diagonal},
for some $\alpha\in(0,1/2)$.
Then for any $T>0$, $\displaystyle \Phi \in C([0,T];H^{s}(\mathbb{T}))$ almost surely.
\label{Prop:ContofStochConv}
\end{proposition}
The proof of Proposition \ref{Prop:ContofStochConv} can be found in the next subsection.  The proof of Theorem \ref{Thm:Contraction} will also require the following Lemma concerning the $X^{s,b}_{T}$ spaces, which allows us to gain a small power of $T$ by raising the temporal exponent $b$.

\begin{lemma}
Let $0<b<\frac{1}{2}$, $s\in\mathbb{R}$, then
\[ \|u\|_{X^{s,b}_{T}}\lesssim T^{\frac{1}{2}-b-}\|u\|_{X^{s,\frac{1}{2}}_{T}}.\]
\label{Lemma:gainpowerofT}
\end{lemma}
\noindent The proof of Lemma \ref{Lemma:gainpowerofT} will be included in this subsection.  Before beginning this proof, we recall the following property of the $X^{s,b}_{T}$ spaces, to be exploited throughout this paper.  For any $b<\frac{1}{2}$, letting $\chi_{[0,T]}$ denote the characteristic function of the interval $[0,T]$, we have
\begin{align}
\|u\|_{X^{s,b}_{T}} \sim  \|\chi_{[0,T]}u\|_{X^{s,b}}.
\label{Eqn:cutoff}
\end{align}
We proceed with the justification of \eqref{Eqn:cutoff}.  First, we show that $\displaystyle\|u\|_{X^{s,b}_{T}}\lesssim \|\chi_{[0,T]}u\|_{X^{s,b}}$.
From the definition of the $X^{s,b}_{T}$ norm, and the fact that $u=\chi_{[0,T]}u$ on $[0,T]$, it suffices to demonstrate that, for $b<\frac{1}{2}$, if $u\in X^{s,b}$ then $\chi_{[0,T]}u \in X^{s,b}$.  We find
\begin{align}
\|\chi_{[0,T]}u\|_{X^{s,b}} &= \|\langle n \rangle^{s}
\langle i(\tau- n^{3}) + n^{2}\rangle^{b}(\chi_{[0,T]}u)^{\sim}(n,\tau)\|_{L^{2}_{n,\tau}}
\notag  \\
&\sim  \|\langle n \rangle^{s}
\langle \tau- n^{3}\rangle^{b}(\chi_{[0,T]}u)^{\sim}(n,\tau)\|_{L^{2}_{n,\tau}} \notag \\
&\ \ \ \ \ \ \ \ \ \ \ \ + \|\langle n \rangle^{s+2b}
(\chi_{[0,T]}u)^{\sim}(n,\tau)\|_{L^{2}_{n,\tau}}.
\label{Eqn:lemma1}
\end{align}
Using boundedness of multiplication by a cutoff function in $H^{b}_{t}$ for $b<\frac{1}{2}$, we have
\begin{align}
\|\langle n \rangle^{s}
\langle \tau- n^{3}\rangle^{b}(\chi_{[0,T]}u)^{\sim}(n,\tau)\|_{L^{2}_{n,\tau}}
&= \|e^{-it\partial_{x}^{3}}\chi_{[0,T]}u\|_{H^{s}_{x}H^{b}_{t}}
 \notag \\
&= \|\chi_{[0,T]}e^{-it\partial_{x}^{3}}u\|_{H^{s}_{x}H^{b}_{t}}
 \notag \\
&\lesssim
\|e^{-it\partial_{x}^{3}}u\|_{H^{s}_{x}H^{b}_{t}}
 \notag \\
&=
\|\langle n \rangle^{s}
\langle \tau- n^{3}\rangle^{b}\tilde{u}(n,\tau)\|_{L^{2}_{n,\tau}}
 \notag \\
&\leq
\|\langle n \rangle^{s}
\langle i(\tau- n^{3}) + n^{2}\rangle^{b}\tilde{u}(n,\tau)\|_{L^{2}_{n,\tau}}
 \notag \\
&= \|u\|_{X^{s,b}}.
\label{Eqn:lemma2}
\end{align}
Then, by Plancherel,
\begin{align}
\|\langle n \rangle^{s+2b} (\chi_{[0,T]}u)^{\sim}(n,\tau)\|_{L^{2}_{n,\tau}}
&=
\|\langle n \rangle^{s+2b}\chi_{[0,T]}\hat{u}(n,t)\|_{L^{2}_{n,t}}
\notag  \\
&\leq
\|\langle n \rangle^{s+2b}
\hat{u}(n,t)\|_{L^{2}_{n,t}}
\notag  \\
&=
\|\langle n \rangle^{s+2b}
\tilde{u}(n,\tau)\|_{L^{2}_{n,\tau}}
\notag  \\
&\leq \|u\|_{X^{s,b}}.
\label{Eqn:lemma3}
\end{align}
Combining \eqref{Eqn:lemma1}, \eqref{Eqn:lemma2} and \eqref{Eqn:lemma3}, we have
\begin{align*}
\|\chi_{[0,T]}u\|_{X^{s,b}} \lesssim \|u\|_{X^{s,b}}.
\end{align*}
To prove \eqref{Eqn:cutoff}, it remains to justify that $\displaystyle\|u\|_{X^{s,b}_{T}}\gtrsim \|\chi_{[0,T]}u\|_{X^{s,b}}$.  Let $w\in X^{s,b}_{T}$ satisfy $w\equiv u$ on $[0,T]$.  Then we have
\begin{align*}
\|\chi_{[0,T]}u\|_{X^{s,b}}=\|\chi_{[0,T]}w\|_{X^{s,b}} \lesssim \|w\|_{X^{s,b}},
\end{align*}
by the argument above.  Thus $\|w\|_{X^{s,b}} \gtrsim
\|\chi_{[0,T]}u\|_{X^{s,b}}$ for all $w$ such that $w\equiv u$ on $[0,T]$, and $\displaystyle\|u\|_{X^{s,b}_{T}}\gtrsim \|\chi_{[0,T]}u\|_{X^{s,b}}$ follows from the definition of
$X^{s,b}_{T}$.  Having established \eqref{Eqn:cutoff}, we proceed with the proof of Lemma \ref{Lemma:gainpowerofT}.

\begin{proof}[Proof of Lemma \ref{Lemma:gainpowerofT}:]
Letting $v=\chi_{[0,T]}u$, by interpolation, we have
\begin{align}
\|v\|_{X^{s,b}} \lesssim  \|v\|_{X^{s,0}}^{1-2b}\|v\|_{X^{s,\frac{1}{2}}}^{2b}.
\label{Eqn:cutoff2}
\end{align}
Next recall that $\widehat{\chi_{[0,T]}}(\tau)=T\widehat{\chi_{[0,1]}}(T\tau)$, which gives
\begin{align*}
\|\widehat{\chi_{[0,T]}}\|_{L^{q}_{\tau}} \sim T^{\frac{q-1}{q}} \|\widehat{\chi_{[0,1]}}\|_{L^{q}_{\tau}} \sim  T^{\frac{q-1}{q}}.
\end{align*}
We can therefore gain a positive power of $T$ as long as $q>1$.  Let $w\in X^{s,b}$ be any function such that $w\equiv u$ on $[0,T]$. For fixed $n$, we apply the Young and H\"{o}lder inequalities to find
\begin{align*}
\|\tilde{v}(n,\tau)\|_{L^{2}_{\tau}} &= \|\widehat{\chi_{[0,T]}}*\tilde{w}(n,\tau)\|_{L^{2}_{\tau}} \\
&\leq  \|\widehat{\chi_{[0,T]}}\|_{L^{2-}_{\tau}}\|\tilde{w}(n,\tau)\|_{L^{1+}_{\tau}} \\
&\lesssim  T^{\frac{1}{2}-}\|\langle i(\tau -n^{3})+n^{2}\rangle^{-\frac{1}{2}}\|_{L^{2+}_{\tau}} \\
&\ \ \ \ \ \ \ \ \ \ \ \|\langle i(\tau -n^{3})+n^{2}\rangle^{\frac{1}{2}}\tilde{w}(n,\tau)
\|_{L^{2}_{\tau}}  \\
&\lesssim  T^{\frac{1}{2}-}\|\langle i(\tau -n^{3})+n^{2}\rangle^{\frac{1}{2}}\tilde{w}(n,\tau)
\|_{L^{2}_{\tau}}.
\end{align*}
This gives
\begin{align}
\|v\|_{X^{s,0}} \lesssim  T^{\frac{1}{2}-}\|w\|_{X^{s,\frac{1}{2}}}.
\label{Eqn:cutoff3}
\end{align}
Combining \eqref{Eqn:cutoff}, \eqref{Eqn:cutoff2} and \eqref{Eqn:cutoff3},
\begin{align*}
\|u\|_{X^{s,b}_{T}} \lesssim  \|v\|_{X^{s,b}}
\lesssim  T^{\frac{1}{2}-b-}\|w\|_{X^{s,\frac{1}{2}}}
\end{align*}
for every $w\in X^{s.b}$ such that $w\equiv u$ on $[0,T]$.
Lemma \ref{Lemma:gainpowerofT} follows from the definition of $X^{s,b}_{T}$.
\end{proof}

The remainder of this Section is organized as follows.  In the next subsection we present the proofs of Propositions \ref{Prop:EstofStochConv} and \ref{Prop:ContofStochConv}.  Then, in subsection 2.3, we will prove Theorem \ref{Thm:Contraction} using Propositions \ref{Prop:Linear} - \ref{Prop:ContofStochConv}, and Lemma \ref{Lemma:gainpowerofT}.


\subsection{Stochastic convolution estimates}

In this subsection we present the proofs of Propositions \ref{Prop:EstofStochConv}
and \ref{Prop:ContofStochConv}.  These Propositions provide estimates on the stochastic convolution to be used in the proof of Theorem \ref{Thm:Contraction} (in the next subsection).  In particular, we will use Propositions \ref{Prop:EstofStochConv} and \ref{Prop:ContofStochConv} to prove the existence and continuity of solutions to \eqref{Eqn:SDKDV}, respectively.

\begin{proof} [Proof of Proposition \ref{Prop:EstofStochConv}]

Using \eqref{Eqn:cutoff}, $b<\frac{1}{2}$, and the Stochastic Fubini Theorem, we compute that
\begin{align*}
\|&\int_{0}^{t}S(t-t')\phi\partial_{x}dW(t')\|_{X^{s,b}_{T}} \\
&\sim
\|\chi_{[0,T]}(t)\int_{0}^{t}S(t-t')\phi\partial_{x}dW(t')\|_{X^{s,b}} \\
&=
\|\chi_{[0,T]}(t)\int_{0}^{t}S(t-t')\chi_{[0,T]}(t')\phi\partial_{x}dW(t')\|_{X^{s,b}} \\
&\lesssim
\|\chi_{[0,\infty)}(t)\int_{0}^{t}S(t-t')\chi_{[0,T]}(t')\phi\partial_{x}dW(t')\|_{X^{s,b}} \\
&= \|\langle n\rangle^{s}\langle i(\tau-n^{3}) + n^{2}\rangle^{b}\int_{0}^{\infty}e^{-it\tau}
\big(\int_{0}^{t}\chi_{[0,T]}(t')e^{-n^{2}(t-t')}e^{in^{3}(t-t')}\phi_{n}(in)dB_{n}(t')\big)dt\|_{L^{2}_{n,\tau}}
\\
&= \||n\phi_{n}|\langle n\rangle^{s}\langle i(\tau-n^{3}) + n^{2}\rangle^{b}\int_{0}^{\infty}\chi_{[0,T]}(t')
\Big(\int_{t'}^{\infty}e^{-it\tau}e^{-n^{2}(t-t')}e^{in^{3}(t-t')}dt\Big)dB_{n}(t')\|_{L^{2}_{n,\tau}} \\
&= \||n\phi_{n}|\langle n\rangle^{s}\langle i\tilde{\tau} + n^{2}\rangle^{b}\int_{0}^{T}
\Big(\int_{t'}^{\infty}e^{-it\tilde{\tau}}e^{-n^{2}(t-t')}e^{-in^{3}t'}dt\Big)dB_{n}(t')\|_{L^{2}_{n,\tilde{\tau}}},
\end{align*}
where, for each fixed $n$, we have taken $\tilde{\tau}=\tau-n^{3}$.  Now we find, for each $n$,
\begin{align*}
\int_{t'}^{\infty}e^{-it\tilde{\tau}}e^{-n^{2}(t-t')}e^{-in^{3}t'}dt &= e^{n^{2}t'}e^{-in^{3}t'}
\Big[\frac{1}{-i\tilde{\tau}-n^{2}}e^{-it\tilde{\tau}}e^{-n^{2}t}\Big]^{t=\infty}_{t=t'} \\
&= \frac{1}{i\tilde{\tau}+n^{2}}e^{-in^{3}t'}e^{-it'\tilde{\tau}}.
\end{align*}
Then, bringing the expectation inside, and applying the It\^{o} isometry,
\begin{align*}
\mathbb{E}\Big(\|\int_{0}^{t}&S(t-t')\phi\partial_{x}dW(t')\|_{X^{s,b}_{T}}^{2}\Big)
\\&\lesssim  \sum_{n}\langle n\rangle^{2s}n^{2}|\phi_{n}|^{2}\int_{-\infty}^{\infty}\langle i\tilde{\tau} + n^{2}\rangle^{2b}
\mathbb{E}\Big(\big|\int_{0}^{T}\Big(\frac{1}{i\tilde{\tau}+n^{2}}e^{-in^{3}t'}e^{-it'\tilde{\tau}}\Big)dB_{n}(t')\big|^{2}
\Big)d\tilde{\tau}  \\
&= \sum_{n}\langle n\rangle^{2s}n^{2}|\phi_{n}|^{2}\int_{-\infty}^{\infty}\langle i\tilde{\tau} + n^{2} \rangle^{2b}
\Big(\int_{0}^{T}\frac{1}{\tilde{\tau}^{2}+n^{4}}dt'\Big)d\tilde{\tau} \\
&= T\sum_{n}\langle n\rangle^{2s}n^{2}|\phi_{n}|^{2}\int_{-\infty}^{\infty}\frac{\langle i\tilde{\tau} + n^{2}\rangle^{2b}}
{\tilde{\tau}^{2}+n^{4}}d\tilde{\tau} \\
&\lesssim  T\sum_{n}\langle n\rangle^{2s}|n|^{4b}|\phi_{n}|^{2}\int_{-\infty}^{\infty}
\frac{1}{\langle\rho\rangle^{2-2b}}d\rho \\
&\sim T\|\phi\|_{H^{s+2b}}^{2}, \ \ \ \text{for}\  b<\frac{1}{2}.
\end{align*}
\end{proof}

\begin{proof} [Proof of Proposition \ref{Prop:ContofStochConv}]

From the identity
\[ \int_{r}^{t}(t-t')^{\alpha-1}(t'-r)^{-\alpha}dt' = \frac{\pi}{\sin\pi\alpha}\]
for $\alpha\in(0,1/2), 0\leq r \leq t' \leq t$.  Letting
\[Y(t')=\int_{0}^{t'}S(t'-r)(t'-r)^{-\alpha}\phi \partial_{x}dW(r),\]
we find
\begin{align}
\frac{\sin(\pi\alpha)}{\pi}\int_{0}^{t}&S(t-t')(t-t')^{\alpha-1}Y(t')dt'
\notag \\
&=  \frac{\sin(\pi\alpha)}{\pi}\int_{0}^{t}S(t-t')(t-t')^{\alpha-1}
\Big(\int_{0}^{t'}S(t'-r)(t'-r)^{-\alpha}\phi
\partial_{x}dW(r)\big)dt'  \notag \\
&=  \frac{\sin(\pi\alpha)}{\pi}\int_{0}^{t}S(t-r) \Big(\int_{r}^{t}(t-t')^{\alpha-1}(t'-r)^{-\alpha}dt'\Big)\phi
\partial_{x}dW(r)  \label{Eqn:semigroup}\\
&=  \int_{0}^{t}S(t-r)\phi \partial_{x}dW(r) \notag \\
&=  \Phi(t). \notag
\end{align}
We have obtained (\ref{Eqn:semigroup}) from
\[S(t-r)=S(t-t')S(t'-r)\]
for $r\leq t' \leq t$.
Next we will use the following Lemma from \cite[Lemma 2.7]{DP}:

\begin{lemma}
Let $T>0$, $\alpha\in(0,1/2)$, and $m>\frac{1}{2\alpha}$.  For $f\in L^{2m}([0,T];H^{s}(\mathbb{T}))$, let
\[F(t)=\int_{0}^{t}S(t-t')(t-t')^{\alpha-1}f(t')dt', \ \ \ 0\leq t \leq T.\]
Then $F\in C([0,T];H^{s}(\mathbb{T}))$.  Moreover $\exists\, C = C(m,T)$ such that
\[\|F(t)\|_{H^{s}(\mathbb{T})} \leq C\|f\|_{L^{2m}([0,T];H^{s}(\mathbb{T}))}.\]
\label{Lemma:DP}
\end{lemma}
\noindent The proof of Lemma \ref{Lemma:DP} only requires boundedness of
\begin{align*}
M_{T} &= \sup_{t\in[0,T]} \|S(t)\|_{H^{s}\rightarrow H^{s}} \\
&= \sup_{t\in[0,T]} \sup_{\|f\|_{H^{s}}=1}\|S(t)f\|_{H^{s}}  \\
&= \sup_{t\in[0,T]} \sup_{\|f\|_{H^{s}}=1} \Big(\sum_{n}\langle n\rangle^{2s} e^{-2n^{2}|t|}|\hat{f}(n)|^{2}\Big)^{1/2} \\
&\leq 1.
\end{align*}
In view of Lemma \ref{Lemma:DP}, to prove Proposition \ref{Prop:ContofStochConv} it suffices to show that
\[Y(t')\in L^{2m}([0,T];H^{s}(\mathbb{T}))\]
almost surely.  Using the It\^{o} isometry, we compute for each $n$,
\begin{align*}
\mathbb{E}\Big(|\widehat{Y(t')}(n)|^{2}\Big) &= \mathbb{E}\Big(|\int_{0}^{t'}e^{-n^{2}(t'-r)}e^{in^{3}(t'-r)}
(t'-r)^{-\alpha}\phi_{n}(in)dB_{n}(r)|^{2}\Big)  \\
&= |\phi_{n}|^{2}n^{2}\int_{0}^{t'}e^{-2n^{2}(t'-r)}(t'-r)^{-2\alpha}dr,
\end{align*}
and we find
\begin{align*}
\int_{0}^{t'}e^{-2n^{2}(t'-r)}(t'-r)^{-2\alpha}dr
&= \int_{0}^{n^{2}t'}e^{-2u}(\frac{u}{n^{2}})^{-2\alpha}\frac{du}{n^{2}} \\
&= \frac{1}{n^{2(1-2\alpha)}}\int_{0}^{n^{2}t'}e^{-2u}u^{-2\alpha}du  \\
&\leq \frac{1}{n^{2(1-2\alpha)}}\int_{0}^{\infty}e^{-2u}u^{-2\alpha}du  \\
&\lesssim \frac{1}{n^{2(1-2\alpha)}},
\end{align*}
for $\alpha <\frac{1}{2}$.  This gives
\[\mathbb{E}\Big(|\widehat{Y(t')}(n)|^{2}\Big) \lesssim |n|^{4\alpha}|\phi_{n}|^{2}.\]
Next we apply the Minkowski integral inequality, with $2m>2$, to find
\begin{align*}
\mathbb{E}\Big(\|Y(t')\|_{H^{s}(\mathbb{T})}^{2m}\Big)
&= \mathbb{E}\Big[\Big(\sum_{n}\langle n\rangle^{2s}|\widehat{Y(t)}(n)|^{2}\Big)^{\frac{2m}{2}}\Big]  \\
&\leq \Big(\sum_{n}\langle n\rangle^{2s}\Big(\mathbb{E}\Big(|\widehat{Y(t')}(n)|^{2m}\Big)\Big)^{\frac{2}{2m}}\Big)
^{\frac{2m}{2}}  \\
&\lesssim  \Big(\sum_{n}\langle n\rangle^{2s}\Big(\langle n\rangle^{4\alpha m}|\phi_{n}|^{2m}\Big)^{\frac{2}{2m}}\Big)
^{\frac{2m}{2}}  \\
&= \Big(\sum_{n}\langle n\rangle^{2s+4\alpha}|\phi_{n}|^{2}\Big)^{m}  \\
&= \|\phi\|_{H^{s+2\alpha}}^{2m}.
\end{align*}
Then
\begin{align*}
\mathbb{E}\Big(\int_{0}^{T}\|Y(t')\|_{H^{s}(\mathbb{T})}^{2m}dt'\Big)
\lesssim \|\phi\|_{H^{s+2\alpha}}^{2m}T < \infty.
\end{align*}
Thus $Y(t')\in L^{2m}([0,T];H^{s}(\mathbb{T}))$
almost surely, and the proof of Proposition \ref{Prop:ContofStochConv} is complete.

\end{proof}

\subsection{Proof of Theorem \ref{Thm:Contraction}}

We will now proceed with the proof of Theorem \ref{Thm:Contraction} using Propositions \ref{Prop:Linear} - \ref{Prop:ContofStochConv} and Lemma \ref{Lemma:gainpowerofT}.  Recall that the proof of Proposition \ref{Prop:bilinear} is included in the appendix.

\begin{proof}[Proof of Theorem \ref{Thm:Contraction}]

As stated in the beginning of this Section, we will prove that the solution $u$ to \eqref{Eqn:SDKDV-Duh-1} is almost surely the unique fixed point of a contraction on $S(t)u_{0} + \Phi(t) + B$, where $B$ is the unit ball in the space $X^{s,b}_{T_{\omega}}$ of space-time functions (adapted to the KdV-Burgers equation), for a stopping time $T_{\omega}>0$, and suitable $s,b \in \mathbb{R}$.  This begins by reexpressing \eqref{Eqn:SDKDV-Duh-1} in terms of the nonlinear part of the solution.  That is,
letting $z(t):=S(t)u_{0}$, we can rewrite (\ref{Eqn:SDKDV-Duh-1}) in terms of $v:=u-z-\Phi$,
\begin{align}
v &= -\frac{1}{2}\int_{0}^{t}S(t-t')\partial_{x}((v+z+\Phi)^{2}(t'))dt' \notag \\
&=: \Gamma(v).
\label{Eqn:SDKDV-Duh-2}
\end{align}
This way, the solution $u$ to
\eqref{Eqn:SDKDV-Duh-1} is the unique fixed point of the contraction $\tilde{\Gamma}(u)= z+\Phi + \Gamma(u-z-\Phi)$
on the unit ball in $X^{s,b}_{T_{\omega}}$ centered at $z+\Phi$ if and only if $v$ is the unique fixed point of the contraction $\Gamma(v)$ on the unit ball in $X^{s,b}_{T_{\omega}}$ centered at $0$.  To simplify presentation, we will prove the latter; that $\Gamma$ is almost surely a contraction on the unit ball in $X^{s,\frac{1}{2}-\varepsilon}_{T_{\omega}}$, for $T_{\omega}>0$ sufficiently small.  That is, we will prove that there is a stopping time $T_{\omega}$, such that, almost surely, $T_{\omega}>0$, and for all $u,v\in X^{s,\frac{1}{2}-\varepsilon}_{T_{\omega}}$ with $\|u\|_{X^{s,\frac{1}{2}-\varepsilon}_{T_{\omega}}},\|v\|_{X^{s,\frac{1}{2}-\varepsilon}_{T_{\omega}}}\leq 1$, we have
\begin{align}
\|\Gamma(v)\|_{X^{s,\frac{1}{2}-\varepsilon}_{T_{\omega}}} &\leq 1,  \label{Eqn:mapsintoball} \\
\|\Gamma(u)-\Gamma(v)\|_{X^{s,\frac{1}{2}-\varepsilon}_{T_{\omega}}} &\leq \frac{1}{2}\|u-v\|_{X^{s,\frac{1}{2}-\varepsilon}_{T_{\omega}}}. \label{Eqn:contracts}
\end{align}
In fact, we will show that for any $T>0$,
\begin{align}
\|\Gamma(v)\|_{X^{s,\frac{1}{2}-\varepsilon}_{T}} &\leq C T^{\varepsilon-}\Big(\|u_{0}\|_{H^{s}(\mathbb{T})}
 +
\|v\|_{X^{s,\frac{1}{2}-\varepsilon}_{T}} + \|\chi_{[0,T]}\Phi\|_{X^{s,\frac{1}{2}-\varepsilon}}
\Big)^{2}, \label{Eqn:mapsinto-tech} \\
\|\Gamma(u)-\Gamma(v)\|_{X^{s,\frac{1}{2}-\varepsilon}_{T}} &\leq CT^{\varepsilon-}\Big(\|u_{0}\|_{H^{s}(\mathbb{T})}
 +
\|u\|_{X^{s,\frac{1}{2}-\varepsilon}_{T}} +
\|v\|_{X^{s,\frac{1}{2}-\varepsilon}_{T}} + \|\chi_{[0,T]}\Phi\|_{X^{s,\frac{1}{2}-\varepsilon}}\Big)
\notag  \\
&\ \ \ \ \ \ \ \ \|u-v\|_{X^{s,\frac{1}{2}-\varepsilon}_{T}}, \label{Eqn:contracts-tech}
\end{align}
for some constant $C>0$.  Then (\ref{Eqn:mapsintoball}) and (\ref{Eqn:contracts}) will follow from (\ref{Eqn:mapsinto-tech}) and (\ref{Eqn:contracts-tech}), respectively, by taking
\begin{align}
T_{\omega} := \min \Big\{T>0:2CT^{\varepsilon-}\Big(\|u_{0}\|_{H^{s}(\mathbb{T})} + 2 + \|\chi_{[0,T]}\Phi\|_{X^{s,\frac{1}{2}-\varepsilon}}\Big)^{2}
\geq 1\Big\}.
\label{Eqn:stoptime}
\end{align}
\noindent  By Proposition \ref{Prop:EstofStochConv},
$$ \mathbb{E}(\|\Phi\|^{2}_{X^{s,\frac{1}{2}-\varepsilon}_{1}}) \lesssim \|\phi\|^{2}_{H^{s+1-2\varepsilon}} <\infty,$$
and we have, for any $0<T<1$,
\begin{align*}
\|\chi_{[0,T]}\Phi\|_{X^{s,\frac{1}{2}-\varepsilon}} &\lesssim \|\Phi\|_{X^{s,\frac{1}{2}-\varepsilon}_{1}} \\
&\leq  C(\omega),
\end{align*}
almost surely.  In addition, since $b < \frac{1}{2}$, $\|\chi_{[0,T]}\Phi\|_{X^{s,\frac{1}{2}-\varepsilon}}$ is almost surely continuous with respect to $T$.  It follows that $T_{\omega}>0$ almost surely.  Since $\|\chi_{[0,T]}\Phi\|_{X^{s,b}}$ is $\mathcal{F}_{T}$-measurable, $T_{\omega}$ is a stopping time.

We proceed to justify \eqref{Eqn:mapsinto-tech} and \eqref{Eqn:contracts-tech}, beginning with (\ref{Eqn:mapsinto-tech}).  By Lemma \ref{Lemma:gainpowerofT}, Proposition \ref{Prop:Nonhom-Linear} and Proposition \ref{Prop:bilinear} we have
\begin{align}
\label{Eqn:mapsinto-1}
\|\Gamma(v)\|_{X^{s,\frac{1}{2}-\varepsilon}_{T}} &\lesssim  T^{\varepsilon-}\|\Gamma(v)\|_{X^{s,\frac{1}{2}}_{T}} \notag \\
&\leq  T^{\varepsilon-} \|\partial_{x}\big((v+z+\Phi)^{2}\big)\|
_{X^{s,-\frac{1}{2}+\gamma}_{T}} \notag \\
&\lesssim   T^{\varepsilon-} \|v+z+\Phi\|_{X^{s,\frac{1}{2}-\varepsilon}_{T}}^{2}
\notag \\
&\lesssim   T^{\varepsilon-} \big( \|v\|_{X^{s,\frac{1}{2}-\varepsilon}_{T}}
+\|z\|_{X^{s,\frac{1}{2}-\varepsilon}_{T}}
+\|\Phi\|_{X^{s,\frac{1}{2}-\varepsilon}_{T}}
\big)^{2} \notag \\
&\sim   T^{\varepsilon-} \big( \|v\|_{X^{s,\frac{1}{2}-\varepsilon}_{T}}
+\|z\|_{X^{s,\frac{1}{2}-\varepsilon}_{T}}
+\|\chi_{[0,T]}\Phi\|_{X^{s,\frac{1}{2}-\varepsilon}}
\big)^{2},  \ \ \text{by \eqref{Eqn:cutoff}}.
\end{align}

\noindent Using Proposition \ref{Prop:Linear},
\begin{align}
\|z\|_{X^{s,\frac{1}{2}-\varepsilon}_{T}} = \|S(t)u_{0}\|_{X^{s,\frac{1}{2}-\varepsilon}_{T}} \leq \|u_{0}\|_{H^{s}(\mathbb{T})}.
\label{Eqn:linbound}
\end{align}

\noindent Then we have by (\ref{Eqn:mapsinto-1}),
\[\|\Gamma(v)\|_{X^{s,\frac{1}{2}-\varepsilon}_{T}} \leq C T^{\varepsilon-}
\Big(\|u_{0}\|_{H^{s}(\mathbb{T})} +  \|v\|_{X^{s,\frac{1}{2}-\varepsilon}_{T}}
+  \|\chi_{[0,T]}\Phi\|_{X^{s,\frac{1}{2}-\varepsilon}}\Big)^{2},\]
and (\ref{Eqn:mapsinto-tech}) holds true.

Turning to \eqref{Eqn:contracts-tech}, we apply Lemma \ref{Lemma:gainpowerofT} and Propositions \ref{Prop:Linear} and \ref{Prop:bilinear} to find
\begin{align*}
\|\Gamma(u)-\Gamma(v)\|_{X^{s,\frac{1}{2}-\varepsilon}_{T}} &= \|\frac{1}{2}\int_{0}^{t}S(t-t')\partial_{x}\big((u+v+z+\Phi)(u-v)\big)dt'\|
_{X^{s,\frac{1}{2}-\varepsilon}_{T}}  \\
&\lesssim  T^{\varepsilon-}\|\int_{0}^{t}S(t-t')\partial_{x}\big((u+v+z+\Phi)(u-v)\big)dt'\|
_{X^{s,\frac{1}{2}}_{T}}  \\
&\lesssim T^{\varepsilon-}\|\partial_{x}\big((u+v+z+\Phi)(u-v)\big)\|_{X^{s,-\frac{1}{2}+\gamma}_{T}} \\
&\lesssim  T^{\varepsilon-}\Big(\|u\|_{X^{s,\frac{1}{2}-\varepsilon}_{T}}
+\|v\|_{X^{s,\frac{1}{2}-\varepsilon}_{T}}
+ \|z\|_{X^{s,\frac{1}{2}-\varepsilon}_{T}}
+\|\Phi\|_{X^{s,\frac{1}{2}-\varepsilon}_{T}}\Big)\\
&\ \ \ \ \ \ \ \ \ \|u-v\|_{X^{s,\frac{1}{2}-\varepsilon}_{T}}
\\
&\sim  T^{\varepsilon-}\Big(\|u\|_{X^{s,\frac{1}{2}-\varepsilon}_{T}}
+\|v\|_{X^{s,\frac{1}{2}-\varepsilon}_{T}} + \|u_{0}\|_{H^{s}(\mathbb{T})} + \|\chi_{[0,T]}\Phi\|_{X^{s,\frac{1}{2}-\varepsilon}}\Big)
\\
&\ \ \ \ \ \ \ \ \ \|u-v\|_{X^{s,\frac{1}{2}-\varepsilon}_{T}}, \ \ \ \text{by \eqref{Eqn:cutoff} and \eqref{Eqn:linbound}}.
\end{align*}
This completes the justification of \eqref{Eqn:contracts-tech}.  Having proven \eqref{Eqn:mapsinto-tech} and \eqref{Eqn:contracts-tech}, \eqref{Eqn:mapsintoball} and \eqref{Eqn:contracts} follow with $T=T_{\omega}$ from the definition \eqref{Eqn:stoptime}.  We conclude that $\Gamma$ is almost surely a contraction on the unit ball in $X^{s,\frac{1}{2}-\varepsilon}_{T_{\omega}}$.  By the Banach fixed point theorem, there is almost surely a unique solution $v\in X^{s,\frac{1}{2}-\varepsilon}_{T_{\omega}}$ to \eqref{Eqn:SDKDV-Duh-2}.  Rewriting \eqref{Eqn:SDKDV-Duh-2} in terms of $u=z+v+\Phi$, we have proven almost sure local existence of a unique solution $u$ to (\ref{Eqn:SDKDV-Duh-1}).

It remains to establish almost sure continuity of $u(t)$ in $H^{s}(\mathbb{T})$, and almost sure continuous dependence on the data.  We begin by proving that $u\in C([0,T];H^{s}(\mathbb{T}))$ almost surely.  Given that $u=z+v+\Phi$, it suffices to verify a.s. continuity of $z$, $v$ and $\Phi$ with separate arguments.  Continuity of $z(t)=S(t)u_{0}$ is trivial.  Almost sure continuity of
$$ v = \Gamma(v) = \int_{0}^{t}S(t-t')\partial_{x}((z+v+\Phi)^{2}(t'))dt'$$ follows from Proposition \ref{Prop:ContofNonlin}, and the following estimate:
\begin{align} \|\partial_{x}\big((z+v+\Phi)^{2}\big)\|_{X^{s,-\frac{1}{2}+\gamma}_{T}}
&\lesssim \|v+z + \Phi\|_{X^{s,\frac{1}{2}-\varepsilon}_{T}}^{2}  \notag \\
&\lesssim (1+\|u_{0}\|_{H^{s}(\mathbb{T})}+C(\omega))^{2}
\notag \\
&< \infty,
\label{Eqn:contbound}
\end{align}
almost surely.  In the statement \eqref{Eqn:contbound} we have invoked Propositions \ref{Prop:Linear}, \ref{Prop:bilinear}, and \ref{Prop:EstofStochConv}.  Finally, almost sure continuity of $\Phi$ follows from Proposition \ref{Prop:ContofStochConv}, by witnessing that $\phi\in HS(L^{2};H^{s+1-2\varepsilon}) \subset HS(L^{2};H^{s+2\alpha})$, when $\alpha>0$ is sufficiently small.

Having established that $u\in C([0,T];H^{s}(\mathbb{T}))$ almost surely, it remains to justify almost sure continuous dependence on the data.
Suppose $\{u^{n}_{0}\}$ is a sequence in $H^{s}(\mathbb{T})$ converging to some $u_{0}$.
From the dependence of the time of local existence $T_{\omega}>0$ on the $H^{s}(\mathbb{T})$ norm of the initial data, it follows that for all $n$ sufficiently large, the solutions $u^{n}$ and $u$ to \eqref{Eqn:SDKDV} with initial data $u_{0}^{n}$ and $u_{0}$, respectively, both exist on a time interval $[0, T_{\omega}]$, with $T_{\omega}>0$ (independent of $n$).  The fixed point method guarantees that the solution map $u_{0}\in H^{s}(\mathbb{T})\mapsto u\in X^{s,\frac{1}{2}-\varepsilon}_{T_{\omega}}$ is analytic.  In particular, we have that $u_{n}\rightarrow u$ in $X_{T_{\omega}}^{s,\frac{1}{2}-\varepsilon}$.
Letting $f_{n}=\partial_{x}\big((u^{n}-u)(u^{n}+u)\big)$, by Proposition \ref{Prop:bilinear}, we have
\begin{align*}
\|f_{n}\|_{X^{s,-\frac{1}{2}+\gamma}_{T_{\omega}}} &=
\|\partial_{x}\big((u^{n}-u)(u^{n}+u)\big)\|_{X^{s,-\frac{1}{2}+\gamma}_{T_{\omega}}}
\\
&\lesssim \|u^{n}+ u\|_{X^{s,\frac{1}{2}-\varepsilon}_{T_{\omega}}}
\|u^{n}-u\|_{X^{s,\frac{1}{2}-\varepsilon}_{T_{\omega}}}
\\
&\longrightarrow 0,
\end{align*}
as $n\rightarrow \infty$.  Writing
$$u^{n} = S(t)u_{0}^{n} + \int_{0}^{t}S(t-t')\partial_{x} \big((u^{n})^{2}\big)dt' + \Phi(t)$$
and
$$u=S(t)u_{0} + \int_{0}^{t}S(t-t')\partial_{x} \big(u^{2}\big)dt' + \Phi(t),$$
we have by Proposition \ref{Prop:ContofNonlin}
that
\begin{align*}
u^{n}-u &= \int_{0}^{t}S(t-t')\partial_{x} \big((u^{n}-u)(u^{n}+u)\big)dt'  \\
&=  \int_{0}^{t}S(t-t')\partial_{x} f(t')dt' \\
&\longrightarrow  0,
\end{align*}
in $C([0,T_{\omega}];H^{s}(\mathbb{T}))$ almost surely.  We conclude that the data to solution map for \eqref{Eqn:SDKDV} is almost surely continuous.

Finally, we observe that the same argument can be used, for fixed $u_{0}\in H^{s}(\mathbb{T})$, to verify that the map $\Phi\in X^{s,\frac{1}{2}-\varepsilon}_{T_{\omega}} \mapsto v=v(\Phi) \in C([0,T_{\omega}];H^{s}(\mathbb{T}))$ satisfying \eqref{Eqn:SDKDV-Duh-2} is almost surely continuous.  In particular, this confirms that the map $\omega\in\Omega \mapsto v=v_{\omega}\in C([0,T_{\omega}];H^{s}(\mathbb{T}))$ satisfying \eqref{Eqn:SDKDV-Duh-2} is $\mathcal{F}_{T_{\omega}}$-measurable, as this is the composition of the measurable map $\omega \in\Omega \mapsto \Phi\in X^{s,\frac{1}{2}-\varepsilon}_{T_{\omega}}$ and the continuous map $\Phi\in X^{s,\frac{1}{2}-\varepsilon}_{T_{\omega}} \mapsto v=v(\Phi) \in C([0,T_{\omega}];H^{s}(\mathbb{T}))$.  Writing $u=z+v+\Phi$,
we combine this observation with Proposition \ref{Prop:ContofStochConv} to conclude that
the map $\omega\in\Omega \mapsto u=u_{\omega}\in C([0,T_{\omega}];H^{s}(\mathbb{T}))$ satisfying \eqref{Eqn:SDKDV-Duh-1} is $\mathcal{F}_{T_{\omega}}$-measurable.

In conclusion, there is a stopping time $T_{\omega}>0$ and a unique process $u\in C([0,T_{\omega}];H^{s}(\mathbb{T}))$ satisfying \eqref{Eqn:SDKDV} on $[0,T_{\omega}]$ almost surely.  The proof of Theorem \ref{Thm:Contraction} is complete.

\end{proof} 

\section{Global well-posedness}
\label{Sec:GWP}

This Section is devoted to the proof of Theorem \ref{Thm:GWP}.  In the first subsection we establish apriori bounds on global-in-time solutions to \eqref{Eqn:SDKDV} which are truncated in spatial frequency.  In the second subsection we establish the convergence needed for the proof of Theorem \ref{Thm:GWP}.

\subsection{Global estimates}

Given $N>0$, let $\mathbb{P}_{N}$ denote the dirichlet projection to $E_{N}=\text{span}\{e^{inx}|0<|n|\leq N\}$.  We consider the frequency truncated stochastic PDE
\begin{align}
\left\{
\begin{array}{ll}du^{N} = \Big(u^{N}_{xx} + u^{N}_{xxx} + \mathbb{P}_{N}\big[((u^{N})^{2})_{x}\big]\Big)dt + \phi^{N}\partial_{x} dW  , \ \
t\geq 0, x\in \mathbb{T}
\\
u^{N}(0,x) = u_{0}^{N}(x) = \mathbb{P}_{N}(u_{0}(x))\in L^{2}(\mathbb{T}),
\end{array} \right.
\label{Eqn:SDKDV-trunc}
\end{align}
where $\phi^{N}=\mathbb{P}_{N}\phi$, and $u^{N}=\mathbb{P}_{N}u^{N}$.
We will solve the Duhamel form of \eqref{Eqn:SDKDV-trunc}
\begin{align}
u^{N} = S(t)u_{0}^{N} + \int_{0}^{t}S(t-t')\mathbb{P}_{N}\big((u^{N})^{2}\big)_{x}dt'
+  \int_{0}^{t}S(t-t')\phi^{N}\partial_{x}dW(t').
\label{Eqn:SDKDV-trunc-Duh}
\end{align}
We will also take
$$ \Phi^{N}(t) = \int_{0}^{t}S(t-t')\phi^{N}\partial_{x}dW(t')$$
to denote the frequency truncated stochastic convolution.
In this subsection, we establish uniform bounds on the solution to \eqref{Eqn:SDKDV-trunc}.
Specifically, we establish the following Propositions.

\begin{proposition}
Let $u_{0}\in H^{s}(\mathbb{T})$, and $\phi\in HS(L^{2}(\mathbb{T}),H^{s}(\mathbb{T}))$ of the form \eqref{Eqn:phi-diagonal}.  For every $N>0$, and each $T>0$, there is almost surely a unique solution $u^{N}(t)$
to \eqref{Eqn:SDKDV-trunc-Duh} for all $t\in[0,T]$.
\label{Prop:finiteGWP}
\end{proposition}
The proof of Proposition \ref{Prop:finiteGWP} is found in the appendix.  Proposition \ref{Prop:finiteGWP} guarantees the global existence of solutions to the frequency truncated stochastic PDE \eqref{Eqn:SDKDV-trunc}.
Because it is finite-dimensionsal, this result is (essentially) independent of any conditions placed on $\phi$; we have taken $\phi\in HS(L^{2},H^{s})$ because this is sufficient for our purposes throughout this paper.

Consider initial data $u_{0}\in L^{2}(\mathbb{T})$,
and additive noise smoothed by $\frac{3}{2}$ spatial derivatives (that is, consider $\phi\in HS(L^{2},H^{1})$).  We can apply Proposition \ref{Prop:finiteGWP} with $s=0$, since $\phi \in HS(L^{2},H^{1})\subset HS(L^{2},L^{2})$.  We conclude that there is a unique solution $u^{N}(t)$ to \eqref{Eqn:SDKDV-trunc} which exists globally in time, almost surely.  With the smoothed noise, we can establish the following bound on the (expected) growth of the $L^{2}$-norm of the solution $u^{N}(t)$.  The crucial point is that this bound is independent of $N$.

\begin{proposition}
Let $u_{0}\in L^{2}(\mathbb{T})$, and $\phi\in HS(L^{2}(\mathbb{T}),H^{1}(\mathbb{T}))$ of the form \eqref{Eqn:phi-diagonal}.
The unique solution $u^{N}(t)$ to \eqref{Eqn:SDKDV-trunc} satisfies
\begin{align}
\mathbb{E}\Big(\sup_{0\leq t \leq T}\|u^{N}(t)\|_{L^{2}_{x}}^{2}\Big) &\leq C,
\label{Eqn:GWPL2bound}
\end{align}
where $C=C(T,\|u^{N}_{0}\|_{L^{2}_{x}},\|\phi^{N}\|_{H^{1}})$.
\label{Prop:GWPL2bound}
\end{proposition}
The proof of Proposition \ref{Prop:GWPL2bound}
will be included in this subsection.  We begin with some computations and lemmata which will be used in the proofs of Propositions \ref{Prop:finiteGWP} and \ref{Prop:GWPL2bound}.  The stochastic PDE (\ref{Eqn:SDKDV-trunc}) is a coupled system of SDEs for the Fourier coefficients $c_{n}(t)$ of $u^{N}(t)$.  For $0<|n|\leq N$, we have
\begin{align}
dc_{n}(t) &= \Big((-n^{2}-in^{3})c_{n}(t) + (in)\sum_{\substack{|n_{1}|\leq N \\ |n_{2}|\leq N \\ n=n_{1}+n_{2}}}
c_{n_{1}}(t)c_{n_{2}}(t)\Big)dt \\ &\ \ \ \ \ \ \ \ \ \ \ + \phi_{n}(in)(dB_{n}^{1}+idB_{n}^{2}) \notag,
\end{align}
where $(B_{n}^{1}(t))_{n\in \mathbb{N}}$, $(B_{n}^{2}(t))_{n\in \mathbb{N}}$ are families of standard real-valued Brownian motions which are mutually independent, and $B_{-n}^{1}=B_{n}^{1}$, $B_{-n}^{2}=-B_{n}^{2}$.
Decomposing $c_{n}(t)=a_{n}(t)+ib_{n}(t)$ into real and imaginary parts, this gives
\begin{align}
da_{n}(t) &= \Big(-n^{2}a_{n}+n^{3}b_{n} + Re\big[(in)\sum_{\substack{|n_{1}|\leq N \\ |n_{2}|\leq N \\ n=n_{1}+n_{2}}}
c_{n_{1}}(t)c_{n_{2}}(t)\big]\Big)dt - n\phi_{n}dB_{n}^{2}  \label{Eqn:df1}\\
db_{n}(t) &= \Big(-n^{2}b_{n}-n^{3}a_{n} + Im\big[(in)\sum_{\substack{|n_{1}|\leq N \\ |n_{2}|\leq N \\ n=n_{1}+n_{2}}}
c_{n_{1}}(t)c_{n_{2}}(t)\big]\Big)dt + n\phi_{n}dB_{n}^{1}. \label{Eqn:df2}
\end{align}
\nopagebreak
Letting $f(a_{n},b_{n})=a_{n}^{2}+b_{n}^{2}$, we have by the It\^{o} formula, \eqref{Eqn:df1}, \eqref{Eqn:df2}, and the property $(dB^{i}_{n})^{2}=dt$ (for $i=1,2$ and every $n$), that
\pagebreak
\begin{align*}
df(a_{n},b_{n}) &= \frac{\partial f}{\partial a_{n}}(a_{n},b_{n})da_{n} + \frac{\partial f}{\partial b_{n}}(a_{n},b_{n})db_{n} +  \frac{1}{2}\Big[\frac{\partial^{2}f}{\partial a_{n}^{2}}\big(-n\phi_{n}dB_{n}^{2}\big)^{2} + \frac{\partial^{2}f}{\partial b_{n}^{2}}\big(n\phi_{n}dB_{n}^{1}\big)^{2}\Big]  \\
&=  2a_{n}\Big[-n^{2}a_{n}+n^{3}b_{n} + Re\big[(in)\sum_{\substack{|n_{1}|\leq N \\ |n_{2}|\leq N \\ n=n_{1}+n_{2}}}
c_{n_{1}}(t)c_{n_{2}}(t)\big]\Big]dt \\ & \ \ \ \ \ \ \ \ \ + 2b_{n}\Big[-n^{2}b_{n}-n^{3}a_{n} + Im\big[(in)\sum_{\substack{|n_{1}|\leq N \\ |n_{2}|\leq N \\ n=n_{1}+n_{2}}}
c_{n_{1}}(t)c_{n_{2}}(t)\big]\Big]dt \\ & \ \ \ \ \ \ \ \ \ -n\phi_{n}a_{n}dB_{n}^{2} + n\phi_{n}b_{n}dB_{n}^{1} + 2n^{2}\phi_{n}^{2}dt  \\
&=  \Big[-2n^{2}(a_{n}^{2}+b_{n}^{2}) + Re\big[(in)a_{n}\sum_{\substack{|n_{1}|\leq N \\ |n_{2}|\leq N \\ n=n_{1}+n_{2}}}
c_{n_{1}}(t)c_{n_{2}}(t)\big] \\
& \ \ \ \ \ \ \ \ \
+ Im\big[(in)b_{n}\sum_{\substack{|n_{1}|\leq N \\ |n_{2}|\leq N \\ n=n_{1}+n_{2}}}
c_{n_{1}}(t)c_{n_{2}}(t)\big]+ 2n^{2}\phi_{n}^{2}\Big]dt \\ & \ \ \ \ \ \ \ \ \  - n\phi_{n}a_{n}dB_{n}^{2} + n\phi_{n}b_{n}dB_{n}^{1}.
\end{align*}

\noindent We have $\|u^{N}\|_{L^{2}_{x}}^{2} = \sum_{|n|\leq N}(a_{n}^{2}+b_{n}^{2})
= \sum_{|n|\leq N}f(a_{n},b_{n})$,
and therefore
\begin{align}
d(\|u^{N}\|^{2}_{L^{2}_{x}}) &= \sum_{|n|\leq N}df(a_{n},b_{n})  \notag \\
&=  \Big[-2\|u^{N}\|^{2}_{\dot{H}^{1}_{x}} + \sum_{|n|\leq N}Re\big[(in)a_{n}\sum_{\substack{|n_{1}|\leq N \\ |n_{2}|\leq N \\ n=n_{1}+n_{2}}}
c_{n_{1}}(t)c_{n_{2}}(t)\big] \notag \\ & \ \ \ \ \ \ \ \ \ + \sum_{|n|\leq N}Im\big[(in)b_{n}\sum_{\substack{|n_{1}|\leq N \\ |n_{2}|\leq N \\ n=n_{1}+n_{2}}}
c_{n_{1}}(t)c_{n_{2}}(t)\big] + 2\|\phi^{N}\|^{2}_{\dot{H}^{1}}\Big]dt \notag \\ & \ \ \ \ \ \ \ \ \   - \sum_{|n|\leq N}n\phi_{n}a_{n}dB_{n}^{2} + \sum_{|n|\leq N}n\phi_{n}b_{n}dB_{n}^{1}.
\label{Eqn:L2growth}
\end{align}
\nopagebreak
\noindent Now we compute that for all $t$,
\pagebreak
\begin{align*}
\sum_{|n|\leq N}\overline{c_{n}(t)}(in)\sum_{\substack{|n_{1}|\leq N \\ |n_{2}|\leq N \\ n=n_{1}+n_{2}}}
c_{n_{1}}(t)c_{n_{2}}(t)  &=  \int_{\mathbb{T}}u^{N}(t)\mathbb{P}_{N}\big(\partial_{x}((u^{N})^{2}(t))\big)dx  \\
&=  \int_{\mathbb{T}}u^{N}(t)\partial_{x}((u^{N})^{2}(t))dx  \\
&=  -\int_{\mathbb{T}}\partial_{x}(u^{N}(t))((u^{N})^{2}(t))dx  \\
&=  -\frac{1}{3}\int_{\mathbb{T}}\partial_{x}((u^{N})^{3}(t))dx \\
&=  0.
\end{align*}

\noindent Then,
\begin{align}
0 &= Re\Big[\sum_{|n|\leq N}\overline{c_{n}(t)}(in)\sum_{\substack{|n_{1}|\leq N \\ |n_{2}|\leq N \\ n=n_{1}+n_{2}}}
c_{n_{1}}(t)c_{n_{2}}(t)\Big]  \notag \\
&= \sum_{|n|\leq N}Re\big[(in)a_{n}\sum_{\substack{|n_{1}|\leq N \\ |n_{2}|\leq N \\ n=n_{1}+n_{2}}}
c_{n_{1}}(t)c_{n_{2}}(t)\big] + \sum_{|n|\leq N}Re\big[-i(in)b_{n}\sum_{\substack{|n_{1}|\leq N \\ |n_{2}|\leq N \\ n=n_{1}+n_{2}}}
c_{n_{1}}(t)c_{n_{2}}(t)\big]  \notag  \\
&= \sum_{|n|\leq N}Re\big[(in)a_{n}\sum_{\substack{|n_{1}|\leq N \\ |n_{2}|\leq N \\ n=n_{1}+n_{2}}}
c_{n_{1}}(t)c_{n_{2}}(t)\big] + \sum_{|n|\leq N}Im\big[(in)b_{n}\sum_{\substack{|n_{1}|\leq N \\ |n_{2}|\leq N \\ n=n_{1}+n_{2}}}
c_{n_{1}}(t)c_{n_{2}}(t)\big].
\label{Eqn:L2growth-cancel}
\end{align}

\noindent Combining \eqref{Eqn:L2growth} and \eqref{Eqn:L2growth-cancel},
\begin{align}
d(\|u^{N}\|^{2}_{L^{2}_{x}}) &= \Big[-2\|u^{N}\|^{2}_{\dot{H}^{1}_{x}} + 2\|\phi^{N}\|^{2}_{\dot{H}^{1}_{x}}\Big]dt - \sum_{|n|\leq N}n\phi_{n}a_{n}dB_{n}^{2} + \sum_{|n|\leq N}n\phi_{n}b_{n}dB_{n}^{1}.
\label{Eqn:dL2}
\end{align}
The proof of Proposition \ref{Prop:GWPL2bound} will be included in this subsection.  For the proof of Proposition \ref{Prop:finiteGWP}, consult the appendix.

\begin{proof}[Proof of Proposition \ref{Prop:GWPL2bound}]


We will use \eqref{Eqn:dL2} to estimate the expected growth of the $L^{2}$-norm of the solution $u^{N}(t)$ to \eqref{Eqn:SDKDV-trunc}.  That is, for any $t>0$, we have by \eqref{Eqn:dL2} that
\begin{align}
\|u^{N}(t)\|_{L^{2}_{x}}^{2} - \|u^{N}(0)\|_{L^{2}_{x}}^{2} &= \int_{0}^{t}d(\|u^{N}(t')\|^{2}_{L^{2}_{x}}) \notag \\
&=  -2\int_{0}^{t}\|u^{N}(t')\|^{2}_{\dot{H}^{1}_{x}}dt' + 2\int_{0}^{t}\|\phi^{N}\|^{2}_{\dot{H}^{1}_{x}}dt' \notag \\ &\ \ \ \ - \sum_{|n|\leq N}n\phi_{n}\int_{0}^{t}a_{n}(t')dB_{n}^{2}(t') + \sum_{|n|\leq N}n\phi_{n}\int_{0}^{t}b_{n}(t')dB_{n}^{1}(t').
\label{Eqn:Itocomp}
\end{align}
Notice that the first term in \eqref{Eqn:Itocomp} is non-positive.  We can drop this term from our estimate as long as it is finite.
We thus proceed to establish that
\begin{align}
\int_{0}^{T}\|u^{N}(t)\|^{2}_{\dot{H}^{1}_{x}}dt \leq
C(\omega) < \infty,
\label{Eqn:drop2}
\end{align}
almost surely.  Letting $v^{N}=u^{N}-S(t)u_{0}^{N}-\Phi^{N}$, we have
\begin{align}
\int_{0}^{T}\|u^{N}(t)\|^{2}_{\dot{H}^{1}_{x}}dt  &\sim  \|u^{N}\|^{2}_{X^{1,0}_{T}} \notag \\
&\leq  \Big( \|S(t)u_{0}^{N}\|_{X^{1,0}_{T}} + \|v^{N}\|_{X^{1,0}_{T}} + \|\Phi^{N}\|_{X^{1,0}_{T}}\Big)^{2}.  \label{Eqn:droptriangle1}
\end{align}

\noindent We compute each term separately.
\begin{align}
\|S(t)u_{0}^{N}\|^{2}_{X^{1,0}_{T}} &\sim \int_{0}^{T}\|S(t)u_{0}^{N}\|^{2}_{\dot{H}^{1}_{x}}dt  \notag\\
&= \int_{0}^{T}\sum_{|n|\leq N, n\neq 0}|n|^{2}|e^{-n^{2}t}e^{in^{3}t}\hat{u_{0}}(n)|^{2}dt  \notag\\
&=  \sum_{|n|\leq N, n\neq 0}|n|^{2}|\hat{u_{0}}(n)|^{2}\int_{0}^{T}
e^{-2n^{2}t}dt  \notag \\
&=  \frac{1}{2}\sum_{|n|\leq N, n\neq 0}|\hat{u_{0}}(n)|^{2}(1-e^{-2n^{2}T})  \notag\\
&\leq  \|u_{0}^{N}\|_{L^{2}_{x}}^{2}.
\label{Eqn:linfinite1}
\end{align}

\noindent Then, using $\langle n^{2}\rangle \leq \langle i(\tau-n^{3}) + n^{2} \rangle $, Propositions 2 and 3, and the definition of $T_{\omega,K}$,
\begin{align}
\|v^{N}\|_{X^{1,0}_{T}} &\leq \|v^{N}\|_{X^{0,\frac{1}{2}}_{T}}  \notag\\
&\leq  T^{\varepsilon-} \|u^{N}\|^{2}_{X^{0,\frac{1}{2}-\varepsilon}_{T}} \notag  \\
&\leq  T^{\varepsilon-} \|u^{N}\|^{2}_{X^{0,\frac{1}{2}-\varepsilon}_{T}} \notag \\
&<  \infty,
\label{Eqn:nonlinfinite1}
\end{align}

\noindent almost surely.  Finally, by (a trivial modification of the proof of) Proposition \ref{Prop:EstofStochConv}, since $\phi^{N}\in HS(L^{2}(\mathbb{T}),H^{1}(\mathbb{T}))$ of the form \eqref{Eqn:phi-diagonal}, we have
$$ \mathbb{E}(\|\Phi^{N}\|^{2}_{X^{1,0}_{T}})
\sim \|\phi^{N}\|^{2}_{\dot{H}^{1}} < \infty,$$
and this gives
\begin{align}
\|\Phi^{N}\|_{X^{1,0}_{T}} \leq C(\omega) < \infty, \label{Eqn:Stochfinite1}
\end{align}
almost surely.  Combining equations (\ref{Eqn:droptriangle1}), (\ref{Eqn:linfinite1}), (\ref{Eqn:nonlinfinite1}) and (\ref{Eqn:Stochfinite1}), the justification of \eqref{Eqn:drop2} is complete.

From \eqref{Eqn:drop2} and \eqref{Eqn:dL2}, we have that, almost surely,
\begin{align*}
\sup_{0\leq t \leq T}\|u^{N}(t)\|_{L^{2}_{x}}^{2} &\leq \|u_{0}\|_{L^{2}_{x}}^{2} + 2T \|\phi\|_{\dot{H}^{1}_{x}}^{2} \\ &\ \ \ + \sum_{|n|\leq N,n\neq 0}|n\phi_{n}|\Big(\sup_{0\leq t \leq T}\Big|\int_{0}^{t}a_{n}(t')dB_{n}^{2}(t')\Big| \\
&\ \ \ \ \ \ \ \ \ \ \ \ \ \ + \sup_{0\leq t \leq T}\Big|\int_{0}^{t}b_{n}(t')dB_{n}^{1}(t')\Big|\Big).
\end{align*}

\noindent For each $n$, the stochastic integrals $X_{n,t}:= \int_{0}^{t}a_{n}(t')dB_{n}^{2}(t')$, $Y_{n,t}:= \int_{0}^{t}b_{n}(t')dB_{n}^{1}(t')$ are continuous martingales.  Using Burkholder's inequality, we have
\begin{align*}
\mathbb{E}\Big(\sup_{0\leq t \leq T}|X_{n,t}|\Big)
&\leq 3\mathbb{E}\Big(\big(\int_{0}^{T}(a_{n}(t))^{2}dt\big)^{\frac{1}{2}}\Big)  \\
&\leq 3\mathbb{E}\Big(\big(\int_{0}^{T}((a_{n}(t))^{2}+ (b_{n}(t))^{2})dt\big)^{\frac{1}{2}}\Big),
\end{align*}

\noindent and the same inequality holds with $Y_{n,t}$ on the left-hand side.  This gives
\begin{align*}
\mathbb{E}\Big(\sup_{0\leq t \leq T}\|u^{N}(t)\|_{L^{2}_{x}}^{2}\Big)
&\leq \mathbb{E}\Big(\sup_{0\leq t \leq T}\|u^{N}(t)\|^{2}_{L^{2}_{x}}\Big) \\
&\leq \|u_{0}\|_{L^{2}_{x}}^{2} + 2T \|\phi\|_{\dot{H}^{1}}^{2} \\
&\ \ \ \ \ + 6\sum_{|n|\leq N,n\neq 0}|n\phi_{n}|\mathbb{E}\Bigg[\big(\int_{0}^{T}((a_{n}(t))^{2}+ (b_{n}(t))^{2})dt\big)^{1/2}\Bigg] \\
&\leq \|u_{0}\|_{L^{2}_{x}}^{2} + 2T \|\phi\|_{\dot{H}^{1}}^{2} + 6\|\phi\|_{\dot{H}^{1}}\Big[\mathbb{E}\big(\int_{0}^{T}
\|u^{N}(t)\|_{L^{2}_{x}}^{2}dt\big)\Big]^{1/2}  \\
&\leq  \|u_{0}\|_{L^{2}_{x}}^{2} + C(T) \|\phi\|_{\dot{H}^{1}}^{2} + \frac{1}{2}\mathbb{E}\Big(\sup_{0\leq t \leq T}\|u^{N}(t)\|_{L^{2}_{x}}^{2}\Big).
\end{align*}

Rearranging this expression,
\begin{align*}
\mathbb{E}\Big(\sup_{0\leq t \leq T}\|u^{N}(t)\|_{L^{2}_{x}}^{2}\Big) &\leq 2\|u_{0}\|_{L^{2}_{x}}^{2} + 2C(T)\|\phi\|_{\dot{H}^{1}_{x}}^{2},
\end{align*}
and \eqref{Eqn:GWPL2bound} holds true.  The proof
of Proposition \ref{Prop:GWPL2bound} is complete.
\end{proof}

\subsection{Proof of Theorem \ref{Thm:GWP}}

In this subsection we will prove Theorem \ref{Thm:GWP} using Propositions \ref{Prop:finiteGWP} and \ref{Prop:GWPL2bound}.  For the proof of Proposition \ref{Prop:finiteGWP}, consult the appendix.

\begin{proof}[Proof of Theorem \ref{Thm:GWP}:]
Given $u_{0}\in L^{2}(\mathbb{T})$, $\phi\in HS(L^{2},H^{1})$ of the form \eqref{Eqn:phi-diagonal}, for every $N>0$ we form the cutoff functions $u_{0}^{N}=\mathbb{P}_{N}u_{0}$, $\phi^{N}=\mathbb{P}_{N}\phi$.  Given $T>0$, by Proposition \ref{Prop:finiteGWP}, a unique solution $u^{N}(t)$ to \eqref{Eqn:SDKDV-trunc-Duh} almost surely exists for $t\in[0,T]$ and satisfies \eqref{Eqn:GWPL2bound}.  Hence, the sequence $\{u^{N}\}_{N\in \mathbb{N}}$ is bounded in $L^{2}(\Omega;L^{\infty}((0,T);L^{2}(\mathbb{T})))$, and we can extract a subsequence which converges weak-* to a limit $\tilde{u}\in L^{2}(\Omega;L^{\infty}((0,T);L^{2}(\mathbb{T})))$ satisfying \eqref{Eqn:GWPL2bound}.  It remains to justify that $\tilde{u}$ satisfies \eqref{Eqn:SDKDV-Duh-1} on $[0,T]$ almost surely.

Letting $z^{N}(t)=S(t)u_{0}^{N}$, and $v^{N} = u^{N}- z^{N} - \Phi^{N}$, then for each $N$, $v^{N}$ satisfies the truncated equation
\begin{align}
v^{N} &= \int_{0}^{t}S(t-t')\mathbb{P}_{N}\Big(\partial_{x}\big((z^{N}+v^{N}+\Phi^{N})^{2}\big)(t')\Big)dt'
\notag \\
&=: \Gamma^{N}(v^{N}).
\label{Eqn:SDKDV-v-trunc}
\end{align}

\noindent From the proof of Theorem \ref{Thm:Contraction}, $\Gamma^{N}$ is almost surely a contraction on a ball of radius 1 in $X^{0,\frac{1}{2}-\varepsilon}_{\tilde{T}}$ for any $\tilde{T}>0$ satisfying
\begin{align}
2C\tilde{T}^{\varepsilon-}\Big(2+\|u_{0}^{N}\|_{L^{2}_{x}}^{2} + \|\chi_{[0,\tilde{T}]}\Phi^{N}\|^{2}_
{X^{0,\frac{1}{2}-\varepsilon}}\Big)^{2} \leq 1,
\label{Eqn:contracts-trunc}
\end{align}
for a certain constant $C>0$.  Let
\begin{align*}
D(\omega) := \sup_{0\leq t\leq T}\|\tilde{u}(t)\|^{2}_{L^{2}_{x}}.
\end{align*}
Then since $\tilde{u}$ satisfies \eqref{Eqn:GWPL2bound} on $[0,T]$, we have that $D(\omega)<\infty$ almost surely.  Consider $\tilde{T}_{\omega}>0$ satisfying
\begin{align}
2C\tilde{T}_{\omega}^{\varepsilon-}\Big(2 + \|u_{0}\|_{L^{2}_{x}} +  D(\omega) + \|\chi_{[0,T]}\Phi\|_{X^{s,\frac{1}{2}-\varepsilon}}\Big)^{2}
\leq  1.
\label{Eqn:GWPtime}
\end{align}
Then for every $N>0$, we have
\begin{align*}
\|u^{N}_{0}\|_{L^{2}_{x}} \leq \|u_{0}\|_{L^{2}_{x}},
\end{align*}
and
\begin{align*}
\|\chi_{[0,\tilde{T}_{\omega}]}\Phi^{N}\|_{X^{s,\frac{1}{2}-\varepsilon}} \leq \|\chi_{[0,T]}\Phi\|_{X^{s,\frac{1}{2}-\varepsilon}}.
\end{align*}
It follows that \eqref{Eqn:contracts-trunc} is satisfied almost surely for every $N>0$ with $\tilde{T}=\tilde{T}_{\omega}$.  Furthermore, we have
$\tilde{T}_{\omega} \leq T_{\omega}$, where $T_{\omega}$ is the time of local existence for the full solution $v$ coming from the proof of Theorem \ref{Thm:Contraction}.
We conclude that $\Gamma^{N}$ and $\Gamma$ are contractions  (for every $N>0$) in $X^{0,\frac{1}{2}-\varepsilon}_{\tilde{T}_{\omega}}$, where $\tilde{T}_{\omega}$ satisfies \eqref{Eqn:GWPtime}.
In particular, a unique solution $v \in X^{0,\frac{1}{2}-\varepsilon}_{\tilde{T}_{\omega}}$ to \eqref{Eqn:SDKDV-Duh-2} almost surely exists.
Furthermore, for each $N>0$, $v^{N}$ and $v$
are the unique fixed points of the contractions
$\Gamma^{N}$ and $\Gamma$, respectively.
Letting $u=S(t)u_{0}+v+\Phi$, then $u$ solves \eqref{Eqn:SDKDV-Duh-1} on $[0,\tilde{T}_{\omega}]$ almost surely, and we find
\begin{align}
\|u^{N}-u\|_{X^{0,\frac{1}{2}-\varepsilon}_{\tilde{T}_{\omega}}} &=  \|S(t)(u_{0}^{N}-u_{0}) + v^{N}-v + \Phi^{N}-\Phi\|_{X^{0,\frac{1}{2}-\varepsilon}_{\tilde{T}_{\omega}}} \notag \\
&\leq \|S(t)(u_{0}^{N}-u_{0})\|_{X^{0,\frac{1}{2}-\varepsilon}_{\tilde{T}_{\omega}}} + \|v^{N}-v\|_{X^{0,\frac{1}{2}-\varepsilon}_{\tilde{T}_{\omega}}} + \|\Phi^{N}-\Phi\|_{X^{0,\frac{1}{2}-\varepsilon}_{\tilde{T}_{\omega}}}.
\label{Eqn:GWPconv}
\end{align}

\noindent Examining each term,
\begin{align*}
\|S(t)(u_{0}^{N}-u_{0})\|_{X^{0,\frac{1}{2}-\varepsilon}_{\tilde{T}_{\omega}}} \rightarrow 0,
\end{align*}
\noindent by construction.  While
\begin{align*}
\|\Phi^{N}-\Phi\|_{X^{0,\frac{1}{2}-\varepsilon}_{\tilde{T}_{\omega}}} =
\bigg\|\int_{0}^{t}S(t-t')\mathbb{P}_{\geq N}(\phi)\partial_{x}dW(t')\bigg\|_{X^{0,\frac{1}{2}-\varepsilon}_{\tilde{T}_{\omega}}} \rightarrow 0,
\end{align*}
\noindent almost surely, by Proposition \ref{Prop:EstofStochConv}.  Lastly,
\begin{align*}
\|v^{N}-v\|_{X^{0,\frac{1}{2}-\varepsilon}_{\tilde{T}_{\omega}}} &= \|\Gamma^{N}(v^{N})-\Gamma(v)\|_{X^{0,\frac{1}{2}-\varepsilon}_{\tilde{T}_{\omega}}} \\ &\leq
\|\Gamma^{N}(v^{N})-\Gamma^{N}(v)\|_{X^{0,\frac{1}{2}-\varepsilon}_{\tilde{T}_{\omega}}} +
\|\Gamma^{N}(v)-\Gamma(v)\|_{X^{0,\frac{1}{2}-\varepsilon}_{\tilde{T}_{\omega}}},
\end{align*}
\noindent and we find
\begin{align*}
\|\Gamma^{N}(v^{N})-\Gamma^{N}(v)\|_{X^{0,\frac{1}{2}-\varepsilon}_{\tilde{T}_{\omega}}} \leq \frac{1}{2} \|v^{N}-v\|_{X^{0,\frac{1}{2}-\varepsilon}_{\tilde{T}_{\omega}}},
\end{align*}
\noindent since $\Gamma^{N}$ is a contraction for each $N>0$.  Then
\begin{align*}
\|\Gamma(v)-\Gamma^{N}(v)\|_{X^{0,\frac{1}{2}-\varepsilon}_{\tilde{T}_{\omega}}} &=  \bigg\|\int_{0}^{t}S(t-t')\mathbb{P}_{\geq N}((v+S(t')u_{0} + \Phi)^{2})dt'\bigg\|_{X^{0,\frac{1}{2}-\varepsilon}_{\tilde{T}_{\omega}}} \\
&= \|\mathbb{P}_{N}(v)\|_{X^{0,\frac{1}{2}-\varepsilon}_{\tilde{T}_{\omega}}}
\\
&\rightarrow  0,
\end{align*}
almost surely, by the almost sure finiteness of $v$ in $X^{0,\frac{1}{2}-\varepsilon}_{\tilde{T}_{\omega}}$. By \eqref{Eqn:GWPconv}, we conclude that
$u^{N}\rightarrow u$ in $X^{0,\frac{1}{2}-\varepsilon}_{\tilde{T}_{\omega}}$, as $N\rightarrow \infty$, almost surely. Moreover, (as in the justification of almost sure continuous dependence of the local solution on the data, see Section \ref{Sec:LWP}) Propositions \ref{Prop:bilinear}, \ref{Prop:EstofStochConv} and \ref{Prop:ContofNonlin} imply that $u^{N}\rightarrow u$ in $C([0,\tilde{T}_{\omega}];L^{2}(\mathbb{T}))$.  We conclude that $u=\tilde{u}$ for $t\in[0,\tilde{T}_{\omega}]$ almost surely.  This gives
\begin{align}
\|u(\tilde{T}_{\omega})\|_{L^{2}_{x}}^{2} \leq \sup_{0\leq t\leq T}\|\tilde{u}(t)\|_{L^{2}_{x}}^{2}=D(\omega),
\label{Eqn:GWP}
\end{align}
\noindent almost surely.  By \eqref{Eqn:GWPtime} we can iterate the argument above on $[\tilde{T}_{\omega},2\tilde{T}_{\omega}]$.  Thus, we have $u=\tilde{u}$ on $[0,T]$ almost surely, and our proof is complete.
\end{proof}

\section{Appendix}
\label{Sec:app}

In the appendix we cover three topics.  First, we present a formal proof of white noise invariance for \eqref{Eqn:SDKDV}.  Second, we prove the  bilinear estimate from Section \ref{Sec:LWP} (Proposition \ref{Prop:bilinear}).  Lastly, we establish the (almost sure) global existence of solutions to the frequency truncated stochastic PDE \eqref{Eqn:SDKDV-trunc} (Proposition \ref{Prop:finiteGWP}).

\subsection{Formal invariance of white noise}

We begin with the construction of mean zero spatial white noise.  Let $u = \sum_{n} \hat{u}_n e^{inx}$ be a real-valued function on $\mathbb{T}$ with mean zero.  That is, we have $\hat{u}_{0} = 0$ and $\hat{u}_{-n} = \overline{\hat{u}_{n}}$.
First, define $\mu_N$ on $\mathbb{C}^{N} \cong \mathbb{R}^{2N}$ (the space of Fourier coefficients)
with the density
\begin{equation}\label{Eqn:WNdensity}
 d \mu_N = Z_N^{-1} e^{- \sum_{n = 1}^N  |\hat{u}_n|^2 } \textstyle \prod_{ n = 1 }^N d \hat{u}_n ,
\end{equation}

\noindent
where $Z_{N} $ is a normalizing constant.  Note that $\mu_{N}$ is the induced probability measure on $\mathbb{C}^{N}$ under the map
\begin{align*}
 \omega\in\Omega \longmapsto (  g_n(\omega)  )_{n = 1}^N \in \mathbb{C}^{N},
\end{align*}

\noindent
where $g_n(\omega)$, $n = 1, \cdots, N$,  are independent standard complex Gaussian random variables.  Next, the map
\begin{align*}
(\hat{u}_{n})_{n = 1}^N \in \mathbb{C}^{N}
\longmapsto
\sum_{0<|n|\leq N}\hat{u}_{n}e^{inx} \in E_{N},
\end{align*}
pushes $\mu_{N}$ forward to a probability measure on
$E_{N}=\text{span}\{e^{inx}: 0< |n|\leq N\}$, and then to a probability measure on  $H^{s}(\mathbb{T})$ by extension.
White noise $\mu$ on $H^{s}(\mathbb{T})$, for $s<-\frac{1}{2}$, can be defined as the weak limit of the sequence of truncated measures $\mu_{N}$; it is the probability measure induced by the map  $\omega \in \Omega \longmapsto u = \sum_{n \neq 0} g_n(\omega) e^{inx} \in H^{s}(\mathbb{T})$,
where $\{g_n(\omega)\}_{n \geq 1}$ are independent standard complex Gaussian random variables,
and $\hat{g}_{-n} = \overline{\hat{g}_{n}}$.

The stochastic PDE \eqref{Eqn:SDKDV} preserves white noise $\mu$ if and only if the solution map $S_{t}:H^{s}(\mathbb{T})\rightarrow H^{s}(\mathbb{T})$, for $s<-\frac{1}{2}$, satisfies $S_{t}^{*}\mu = \mu$ (in distribution) for each $t\geq 0$.  To be clear, we do not know that the solution map $S_{t}$ is well-defined; Theorem \ref{Thm:Contraction} is designed as a progress towards this definition.
Instead, consider the frequency truncated stochastic PDE \eqref{Eqn:SDKDV-trunc}, with solution map $S_{t}^{N}:
E_{N}\rightarrow
E_{N}$.  In this subsection, we establish the following claim.

\begin{claim}
With $\phi =\textup{Id}$, the equation \eqref{Eqn:SDKDV-trunc} preserves the truncated white noise $\mu_{N}$.  That is,  $(S_{t}^{N})^{*}\mu_{N}=\mu_{N}$ (in distribution) for each $t\geq 0$.
\label{Claim:formalWN}
\end{claim}
This provides formal evidence that the full stochastic PDE, \eqref{Eqn:SDKDV}, preserves spatial white noise, but the lack of well-defined dynamics for \eqref{Eqn:SDKDV} (in the infinite limit) represents an obstruction to the rigorous proof of invariance.

\begin{proof}[Proof of Claim \ref{Claim:formalWN}]

The proof of Claim \ref{Claim:formalWN} is based on a decomposition of \eqref{Eqn:SDKDV-trunc} into the (truncated) KdV equation, plus a rescaled Ornstein-Uhlenbeck process at each spatial frequency.  That is, \eqref{Eqn:SDKDV-trunc} is given by
\begin{align}
du^{N} &= \Bigg[u^{N}_{xx}-u^{N}_{xxx} + \mathbb{P}_{N}\bigg(\partial_{x}\big((u^{N})^{2}\big)\bigg)\Bigg]dt + \mathbb{P}_{N}\partial_{x}dW
\notag \\
&= \underbrace{-u^{N}_{xxx}dt + \mathbb{P}_{N}\bigg(\partial_{x}\big((u^{N})^{2}\big)\bigg)dt }_{\text{truncated KdV}} +  \underbrace{u^{N}_{xx}dt + \mathbb{P}_{N}\partial_{x}dW}_{\text{OU-processes}}.
\label{Eqn:SDKDV-decomp}
\end{align}
The truncated KdV equation preserves spatial white noise (in fact, the full KdV equation preserves spatial white noise, see \cite{OH,OQV,QV}), and an Ornstein-Uhlenbeck process leaves the normal distribution invariant.  We will combine these facts with \eqref{Eqn:SDKDV-decomp} to  prove Claim \ref{Claim:formalWN}.

The stochastic PDE \eqref{Eqn:SDKDV-trunc} is a system of stochastic differential equations in $\mathbb{C}^{N}$, the space of real and complex Fourier coefficients $(a_{n}+ib_{n})_{0<n\leq N}$ of the solution $u^{N}$.  This system is given by \eqref{Eqn:df1} and \eqref{Eqn:df2} (recall that the negative Fourier coefficients are determined by the preservation of reality: $a_{-n}
=a_{n}$, $b_{-n}=-b_{n}$).
By Proposition \ref{Prop:finiteGWP},
the flow map $S_{t}^{N}:E_{N}\rightarrow E_{N}$ for \eqref{Eqn:SDKDV-trunc}
is (almost surely) well-defined for all $t\geq 0$.  Then \eqref{Eqn:SDKDV-trunc}
preserves truncated white noise $\mu_{N}$ (ie. $(S_{t}^{N})^{*}\mu_{N}
=\mu_{N}$ for all $t\geq0$) if and only if the corresponding infinitesimal generator $\mathcal{L}^{N}:C^{\infty}
(\mathbb{C}^{N})\rightarrow \mathbb{R}$ satisfies $(\mathcal{L}^{N})^{*}\mu_{N} =0$, which means that
$\int_{\mathbb{C}^{N}}
\mathcal{L}^{N}f(x)d\mu_{N}(x)=0$ for all $f\in C^{\infty}(\mathbb{C}^{N})$.

With \eqref{Eqn:SDKDV-decomp} in mind, the generator $\mathcal{L}^{N}$ of the system \eqref{Eqn:SDKDV-trunc} can be written as $\mathcal{L}^{N}
=\mathcal{L}_{1}^{N}+\mathcal{L}_{2}^{N}$, where
\begin{align*}
\mathcal{L}_{1}^{N}f(a_{1},\ldots,b_{N})
= \sum_{0<n \leq N} F_{N}(a_{1},\ldots,b_{N})\cdot\nabla f(a_{1},\ldots,b_{N}),
\end{align*}
and
\begin{align*}
\mathcal{L}_{2}^{N}f(a_{1},\ldots,b_{N}) = \sum_{0<n\leq N}-n^{2}\Big(a_{n}
\frac{\partial}{\partial a_{n}}+b_{n}\frac{\partial}{\partial b_{n}} +
\frac{\partial^{2}}{\partial a_{n}^{2}}+\frac{\partial^{2}}{\partial a_{n}^{2}}\Big) f(a_{1},\ldots,b_{N}).
\end{align*}
Here
\begin{align*}
F_{N}(a_{1},\ldots,b_{N}) =
(h_{1}(a_{1},\ldots,b_{N}),\ldots,h_{2N}(a_{1},\ldots,b_{N})),
\end{align*}
with
\begin{align*}
h_{2k}(a_{1},\ldots,b_{N}) = (2k)^{3}b_{2k} +
\textit{Re}\big[(i2k)\sum_{\substack{|n_{1}|\leq N \\ |n_{2}|\leq N \\ 2k=n_{1}+n_{2}}}
(a_{n_{1}}+i b_{n_{1}})(a_{n_{2}}+i b_{n_{2}})\big], \\
h_{2k+1}(a_{1},\ldots,b_{N})
= -(2k)^{3}a_{2k}
+
\textit{Im}\big[(i2k)\sum_{\substack{|n_{1}|\leq N \\ |n_{2}|\leq N \\ 2k=n_{1}+n_{2}}}
(a_{n_{1}}+i b_{n_{1}})(a_{n_{2}}+i b_{n_{2}})\big],
\end{align*}
for $k=1,\ldots,N$.

Notice that $\mathcal{L}_{1}^{N}$ is the generator for the frequency truncated KdV equation
\begin{align}
u^{N}_{t}=-u^{N}_{xxx} + \partial_{x}\mathbb{P}_{N}((u^{N})^{2}).
\label{Eqn:KdVtrunc}
\end{align}
The equation \eqref{Eqn:KdVtrunc} is a Hamiltonian system in the space of Fourier coefficients, which, by Liousville's theorem, preserves the $2N$-dimensional Lebesgue measure $\prod_{n=1}^{N}d\hat{u}_{n}$.  Furthermore, the flow of \eqref{Eqn:KdVtrunc} leaves the $L^{2}$-norm of $u^{N}(t)$ invariant.  By \eqref{Eqn:WNdensity}, the truncated KdV equation  \eqref{Eqn:KdVtrunc} preserves finite-dimensional white noise $\mu_{N}$.  Equivalently, $(\mathcal{L}_{1}^{N})^{*}(\mu_{N})=0$ - ie. $\int\mathcal{L}_{1}^{N}f(x)d\mu_{N}(x)=0$ for all smooth $f$ on $\mathbb{C}^{N}$.

Next observe that $\mathcal{L}_{2}^{N}$ is the generator of an Ornstein-Uhlenbeck process at each spatial frequency.  That is, $\mathcal{L}^{N}_{2}$ is the generator of $2N$ decoupled stochastic differential equations, given by
\begin{align}
da_{n}&=(-n^{2}a_{n})dt - ndB^{2}_{n}, \notag \\
db_{n}&= (-n^{2}b_{n})dt + ndB^{1}_{n},
\label{Eqn:OU}
\end{align}
for $n=1,\ldots,N$.  Again, the negative Fourier coefficients are determined by preserving reality.  The SDE \eqref{Eqn:OU} has an explicit solution,
\begin{align}
a_{n}(t)=a_{n,0}e^{-n^{2}t} - \int_{0}^{t}e^{-n^{2}(t-s)}dB_{n}^{2} \notag \\
b_{n}(t)=b_{n,0}e^{-n^{2}t} + \int_{0}^{t}e^{-n^{2}(t-s)}dB_{n}^{1}.
\label{Eqn:OUsol}
\end{align}
Considering truncated white noise $\mu_{N}$ as initial data corresponds to taking $a_{n,0}, b_{n,0}$ i.i.d. $\mathcal{N}(0,\frac{1}{\sqrt{2}})$.  With these initial distributions, for each $t\geq 0$, and every $1\leq n\leq N$, the solutions $a_{n}(t)$ and $b_{n}(t)$ to \eqref{Eqn:OU} are i.i.d. $\mathcal{N}(0,\frac{1}{\sqrt{2}})$.  Indeed, from \eqref{Eqn:OUsol} we have that $a_{n}(t), b_{n}(t)$ are independent Gaussian processes with mean zero for each $t\geq 0$.  Then by the It\^{o} isometry,
\begin{align*}
\mathbb{E}((a_{n}(t))^{2}) &= \mathbb{E}((a_{n,0})^{2})e^{-2n^{2}t} +
\mathbb{E}((\int_{0}^{t}e^{-n^{2}(t-s)}dB_{n}^{2})^{2})
\\
&= \frac{1}{2}e^{-2n^{2}t} + \int_{0}^{t}e^{-2n^{2}(t-s)}ds \\
&= \frac{1}{2},
\end{align*}
and the same computation applies with $b_{n}(t)$.  We conclude that truncated white noise is invariant under the stochastic process \eqref{Eqn:OU}.  That is, $(\mathcal{L}_{2}^{N})^{*}\mu_{N}=0$.

We conclude that for every smooth $f$ on $\mathbb{C}^{N}$, $\int\mathcal{L}^{N}f(x)d\mu_{N}(x)=\int\mathcal{L}_{1}^{N}f(x)d\mu_{N}(x) + \int\mathcal{L}_{2}^{N}f(x)d\mu_{N}(x)=0$.  That is, $(\mathcal{L}^{N})^{*}\mu_{N}=0$, and truncated white noise $\mu_{N}$ is invariant under the flow of \eqref{Eqn:SDKDV-trunc}, for each $N>0$.  This argument does not pass over to the limit as $N\rightarrow \infty$.  Nonetheless, we have arrived at a formal justification of white noise invariance for \eqref{Eqn:SDKDV}.
\end{proof}

\subsection{Bilinear estimate}
\begin{proof}[Proof of Proposition \ref{Prop:bilinear}]

By \eqref{Eqn:cutoff}, the estimate \eqref{Eqn:bilinear} follows from
\begin{align*}
\|\chi_{[0,T]}(uv)_{x}\|_{X^{s,-\frac{1}{2}+\gamma}}
\lesssim \|\chi_{[0,T]}u\|_{X^{s,\frac{1}{2}-\varepsilon}}\|\chi_{[0,T]}v\|_{X^{s,\frac{1}{2}-\varepsilon}}.
\end{align*}
Then since $\chi_{[0,T]}\partial_{x}(uv) = \partial_{x}\big((\chi_{[0,T]}u)(\chi_{[0,T]}v)\big)$, it suffices to show that
\begin{align}
\|\partial_{x}\big((\chi_{[0,T]}u)(\chi_{[0,T]}v)\big)\|_{X^{s,-\frac{1}{2}+\gamma}}
\lesssim \|\chi_{[0,T]}u\|_{X^{s,\frac{1}{2}-\varepsilon}}\|\chi_{[0,T]}v\|_{X^{s,\frac{1}{2}-\varepsilon}}.
\label{Eqn:bilinear5}
\end{align}
Letting
\begin{align*}
f(n,\tau)&=\langle n\rangle^{s}\langle i(\tau-n^{3}) + n^{2}\rangle^{\frac{1}{2}-\varepsilon} (\widehat{\chi_{[0,T]}u})(n,\tau)  \\
g(n,\tau)&=\langle n\rangle^{s}\langle i(\tau-n^{3}) + n^{2}\rangle^{\frac{1}{2}-\varepsilon} (\widehat{\chi_{[0,T]}v})(n,\tau),
\end{align*}
the estimate (\ref{Eqn:bilinear5}) is equivalent to
\begin{align}
\|\beta(f,g)\|_{L^{2}_{n,\tau}} \lesssim \|f\|_{L^{2}_{n,\tau}}\|g\|_{L^{2}_{n,\tau}},
\label{Eqn:bilinear2}
\end{align}
where
\[
\beta(f,g)(n,\tau) :=\sum_{\substack{n_{1} \\n=n_{1}+n_{2}}} \int_{\tau=\tau_{1}+\tau_{2}}
\frac{|n|\langle n\rangle^{s}f(n_{1},\tau_{1})g(n_{2},\tau_{2})}
{\langle n_{1}\rangle^{s}\langle n_{2}\rangle^{s}\langle \sigma_{0}\rangle^{\frac{1}{2}-\gamma}
\langle \sigma_{1}\rangle^{\frac{1}{2}-\varepsilon
}\langle \sigma_{2}\rangle^{\frac{1}{2}-\varepsilon}}d\tau_{1}
\]
with $\sigma_{i}=i(\tau_{i}-n_{i}^{3}) + n_{i}^{2}$, for $i=1,2$, and $\sigma_{0}=i(\tau-n^{3}) + n^{2}$.

We proceed to justify \eqref{Eqn:bilinear2}.
In the analysis that follows, for a given $(n,n_{1},n_{2})\in \mathbb{Z}^{3}$, we order the magnitudes of the frequencies $|n|,|n_{1}|,|n_{2}|$
from largest to smallest, using capital letters and superscripts to denote the corresponding dyadic shell: ie, $N^{1}\geq N^{2}\geq N^{3}$.  We begin by performing the following computation, which will simplify subsequent estimates.
\begin{lemma}
If $n=n_{1}+n_{2}$ and $s\geq -\frac{1}{2}-\varepsilon$, then
\[\displaystyle\frac{|n|\langle n\rangle^{s}}{\langle n_{1}\rangle^{s}\langle n_{2}\rangle^{s}
\langle nn_{1}n_{2}\rangle^{\frac{1}{2}-\varepsilon}} \lesssim  |N^{1}|^{4\varepsilon}\].
\label{Lemma:kernelbound}
\end{lemma}

\noindent We will also require the following Calculus inequalities:

\begin{lemma}
Let $0<\delta_{1}\leq\delta_{2}$ satisfy $\delta_{1}+\delta_{2}>1$, and let $a\in\mathbb{R}$, then
\[ \int_{-\infty}^{\infty}\frac{d\theta}{\langle\theta\rangle^{\delta_{1}}\langle a-\theta\rangle^{\delta_{2}}}
 \lesssim  \frac{1}{\langle a\rangle^{\alpha}},\]
where $\alpha=\delta_{1}-(1-\delta_{2})_{+}$.  Recall that $(\lambda)_{+}:= \lambda$ if $\lambda>0$, $= \varepsilon>0$ if $\lambda=0$, and $=0$ if $\lambda<0$.
\label{Lemma:integralineq}
\end{lemma}

\begin{lemma}
Let $\delta>\frac{1}{2}$, $n\neq0$, then
\[\Big\|\sum_{n_{1}\neq 0, n_{1}\neq n}\frac{1}{(1+|\mu-n_{1}(n-n_{1})|)^{\delta}}\Big\|_{L^{\infty}_{\mu,n}}<c<\infty.\]
\label{Lemma:supineq}
\end{lemma}
\noindent Proofs of Lemmas \ref{Lemma:integralineq} and \ref{Lemma:supineq} can be found in \cite{GTV} and \cite{KPV2}.  We proceed with the proof of Lemma
\ref{Lemma:kernelbound}.

\begin{proof}[Proof of Lemma \ref{Lemma:kernelbound}]
We consider two cases depending on the relative sizes of $n,n_{1},n_{2}$.  Recall that as $n=n_{1}+n_{2}$,
we always have $N^{1}\sim N^{2}$.

\vspace{0.1in}
\noindent $\bullet$ \textbf{Case 1:} $|n|\sim N^{1}$. \newline
Then
\begin{align*}
\frac{|n|\langle n\rangle^{s}}{\langle n_{1}\rangle^{s}\langle n_{2}\rangle^{s}\langle n n_{1}n_{2}\rangle
^{\frac{1}{2}-\varepsilon}}
&\sim \frac{|N^{1}|^{2\varepsilon}}{|N^{3}|^{s+\frac{1}{2} -\varepsilon}} \\
&\leq |N^{1}|^{2\varepsilon}|N^{3}|^{2\varepsilon}  \ \ \text{since}\ s\geq -\frac{1}{2}-\varepsilon,  \\
&\leq |N^{1}|^{4\varepsilon}.
\end{align*}

\vspace{0.1in}
\noindent $\bullet$ \textbf{Case 2:} $|n|\sim N^{3}$,
so that $|n_{1}|\sim |n_{2}|\sim N^{1}$.
\newline
Then
\begin{align*}
\frac{|n|\langle n\rangle^{s}}{\langle n_{1}\rangle^{s}\langle n_{2}\rangle^{s}\langle n n_{1}n_{2}
\rangle^{\frac{1}{2}-\varepsilon}}
&\sim \Big|\frac{N^{3}}{N^{1}}\Big|^{s+\frac{1}{2}}\frac{|N^{3}|^{\varepsilon}}
{|N^{1}|^{s+\frac{1}{2}-2\varepsilon}} \\
&\leq \Big|\frac{N^{3}}{N^{1}}\Big|^{s+\frac{1}{2}+\varepsilon}
\frac{|N^{1}|^{4\varepsilon}}{|N^{1}|^{s+\frac{1}{2} + \varepsilon}}  \\
&\leq |N^{1}|^{4\varepsilon},   \ \ \text{since}\ s\geq -\frac{1}{2}-\varepsilon.
\end{align*}
This completes the Proof of Lemma \ref{Lemma:kernelbound}.

\end{proof}

We now turn to the proof of (\ref{Eqn:bilinear2}).  In the estimates that follow, we will take $\gamma=\varepsilon$ for simplicity.  For $i=0,1,2$, let
\[ A_{i}=\{(n,n_{1},n_{2},\tau,\tau_{1},\tau_{2})\in
\mathbb{Z}^{3}\times\mathbb{R}^{3}:
\max(|\sigma_{0}|,|\sigma_{1}|,|\sigma_{2}|)=|\sigma_{i}|\},\]
and let $\beta_{i}(f,g)$ denote the contribution to $\beta(f,g)$ coming from $A_{i}$.
We separately estimate each $\beta_{i}(f,g)$, $i=0,1,2$.

\vspace{0.1in}
\noindent $\bullet$ \textbf{Case 1:}
$\max(|\sigma_{0}|,|\sigma_{1}|,|\sigma_{2}|)=|\sigma_{0}|$.
\newline

From the algebraic relation $\max(|\sigma_{0}|,|\sigma_{1}|,|\sigma_{2}|) \geq |nn_{1}n_{2}|$, we have
\begin{align}
\|\beta_{0}(f,g)\|_{L^{2}_{n,\tau}} &\leq \Big\| \sum_{\substack{n_{1} \\n=n_{1}+n_{2}}} \int_{\tau=\tau_{1}+\tau_{2}}
\frac{|n|\langle n\rangle^{s}|f(n_{1},\tau_{1})||g(n_{2},\tau_{2})|}{\langle n_{1}\rangle^{s}\langle n_{2}\rangle^{s}
\langle n n_{1}n_{2}\rangle^{\frac{1}{2}-\varepsilon}\langle \sigma_{1}\rangle^{\frac{1}{2}-\varepsilon}\langle
\sigma_{2}\rangle^{\frac{1}{2}-\varepsilon
}}d\tau_{1}
\Big\|_{L^{2}_{n,\tau}} \notag \\
&\leq \Big\| \sum_{\substack{n_{1} \\n=n_{1}+n_{2}}} \int_{\tau=\tau_{1}+\tau_{2}}
\frac{|N^{1}|^{4\varepsilon}|f(n_{1},\tau_{1})||g(n_{2},\tau_{2})|}
{\langle \sigma_{1}\rangle^{\frac{1}{2}-\varepsilon}\langle \sigma_{2}\rangle^{\frac{1}{2}-\varepsilon}}d\tau_{1}
\Big\|_{L^{2}_{n,\tau}}
\label{Eqn:case1bound}
\end{align}
by Lemma \ref{Lemma:kernelbound}.
We will cancel the factor of $|N^{1}|^{4\varepsilon}$ by considering the relative sizes of
$n_{1},n_{2}$.

\vspace{0.1in}
\noindent $\bullet$ \textbf{Case 1(A):}  $|n_{1}|\sim N^{1}$.
\newline
Then $\displaystyle|\sigma_{1}|\geq n_{1}^{2} \sim |N^{1}|^{2},\ \
\Rightarrow \frac{1}{\langle \sigma_{1}\rangle^{2\varepsilon}}
\lesssim \frac{1}{|N^{1}|^{4\varepsilon}}$.
With this inequality, followed by Cauchy-Schwarz, we find

\begin{align*}
\|\beta_{0}(f,g)\|_{L^{2}_{n,\tau}} &\leq  \Big\| \sum_{\substack{n_{1} \\n=n_{1}+n_{2}}} \int_{\tau=\tau_{1}+\tau_{2}}
\frac{|f(n_{1},\tau_{1})||g(n_{2},\tau_{2})|}{\langle \sigma_{1}\rangle
^{\frac{1}{2}-3\varepsilon}\langle \sigma_{2}\rangle^{\frac{1}{2}-\varepsilon}}d\tau_{1}
\Big\|_{L^{2}_{n,\tau}} \\
&\leq \Big\| \Big(\sum_{\substack{n_{1} \\n=n_{1}+n_{2}}} \int_{\tau=\tau_{1}+\tau_{2}}
|f(n_{1},\tau_{1})|^{2}|g(n_{2},\tau_{2})|^{2}d\tau_{1}\Big)^{\frac{1}{2}} \\
&\ \ \ \ \ \ \ \ \ \ \ \ \ \cdot\Big(\sum_{\substack{n_{1} \\n=n_{1}+n_{2}}} \int_{\tau=\tau_{1}+\tau_{2}}
\frac{d\tau_{1}}{\langle \sigma_{1}\rangle^{1-6\varepsilon}
\langle \sigma_{2}\rangle^{1-2\varepsilon}}\Big)^{\frac{1}{2}}\Big\|_{L^{2}_{n,\tau}}.
\end{align*}

Now observe that

\begin{align*}
\int_{\tau=\tau_{1}+\tau_{2}}&
\frac{d\tau_{1}}{\langle \sigma_{1}\rangle^{1-6\varepsilon}
\langle \sigma_{2}\rangle^{1-2\varepsilon}}  \\
&= \int_{-\infty}^{\infty}\frac{d\tau_{1}}{\langle i(\tau_{1}-n_{1}^{3}) + n_{1}^{2}\rangle^{1-6\varepsilon} \langle i(\tau-\tau_{1}-(n-n_{1})^{3}) + (n-n_{1})^{2}\rangle^{1-2\varepsilon}} \\
&\leq \int_{-\infty}^{\infty}\frac{d\tau_{1}}{\langle \tau_{1}-n_{1}^{3}\rangle^{1-6\varepsilon} \langle
\tau-\tau_{1}-(n-n_{1})^{3}\rangle^{1-2\varepsilon}}  \ \ \ \ \\
&=  \int_{-\infty}^{\infty}\frac{d\tau_{1}}{\langle \theta\rangle^{1-6\varepsilon}
\langle a-\theta \rangle^{1-2\varepsilon}}, \ \ \ \ \ \text{with} \ \ \theta=\tau_{1}-n_{1}^{3}, \ \
a=\tau-n^{3}+3nn_{1}(n-n_{1}), \\
&\leq \frac{1}{\langle a\rangle^{1-8\varepsilon}},
\end{align*}

by Lemma \ref{Lemma:integralineq}.
Letting

\[M_{n,\tau}:=\Big(\sum_{n_{1},n_{1}\neq 0, n_{1}\neq n}\frac{1}{\langle \tau - n^{3} + 3n n_{1}(n-n_{1})
\rangle^{1-8\varepsilon}} \Big)^{\frac{1}{2}}.\]

For $n\neq 0$  we can let
$\mu = \frac{1}{3}\big(\frac{\tau}{n}-n^{2}\big)$, and find
\[ \frac{1}{\langle \tau - n^{3} + 3n n_{1}(n-n_{1})\rangle} \leq \frac{1}{\langle\mu -n_{1}(n-n_{1})\rangle}. \]

This leads to
\begin{align}
\sup_{\substack{n,\tau \\ n\neq 0}}M_{n,\tau} &=
\sup_{\substack{n,\tau \\ n\neq 0}}\Big(\sum_{\substack{n_{1}\\n_{1}\neq0,n_{1}\neq n }}
\frac{1}{\langle \tau - n^{3} + 3n n_{1}(n-n_{1})\rangle^{1-8\varepsilon}}\Big)^{\frac{1}{2}} \notag\\
&\leq \sup_{\substack{n,\mu \\ n\neq 0}}\Big(\sum_{\substack{n_{1}\\n_{1}\neq0,n_{1}\neq n }}
\frac{1}{\langle \mu - n_{1}(n-n_{1})\rangle^{1-8\varepsilon}}\Big)^{\frac{1}{2}} \notag\\
&= C < \infty, \label{Eqn:supbdd}
\end{align}

\noindent by Lemma \ref{Lemma:supineq}, if $\varepsilon<\frac{1}{16}$.

Returning to the proof of (\ref{Eqn:bilinear2}), we pull out $\sup_{n,\tau}M_{n,\tau}$, apply the estimate
(\ref{Eqn:supbdd}), and find
\begin{align*}
\|\beta_{0}(f,g)\|_{L^{2}_{n,\tau}} &\lesssim  \Big(\sup_{\substack{n,\tau\\ n\neq 0}}M_{n,\tau}\Big)\Big\|
\Big(\sum_{\substack{n_{1} \\n=n_{1}+n_{2}}} \int_{\tau=\tau_{1}+\tau_{2}}
|f(n_{1},\tau_{1})|^{2}|g(n_{2},\tau_{2})|^{2}d\tau_{1}\Big)^{\frac{1}{2}}\Big\|_{L^{2}_{n,\tau}} \\
&\lesssim   \Big\| \Big(\sum_{\substack{n_{1} \\n=n_{1}+n_{2}}} \int_{\tau=\tau_{1}+\tau_{2}}
|f(n_{1},\tau_{1})|^{2}|g(n_{2},\tau_{2})|^{2}d\tau_{1}\Big)^{\frac{1}{2}}\Big\|_{L^{2}_{n,\tau}}  \\
&= \|f\|_{L^{2}_{n,\tau}}\|g\|_{L^{2}_{n,\tau}} \\
\end{align*}
\noindent by Fubini's Theorem.  This completes the analysis of Case 1(A).

\vspace{0.1in}
\noindent $\bullet$ \textbf{Case 1(B):}  $|n|\sim |n_{2}| \sim N^{1}$, $|n_{1}|\sim N^{3}$.
\newline

Then $\displaystyle|\sigma_{2}|\gtrsim |N^{1}|^{2},\ \  \Rightarrow \frac{1}{\langle \sigma_{2}\rangle^{2\varepsilon}} \lesssim  \frac{1}{|N^{1}|^{4\varepsilon}}$.  With \eqref{Eqn:case1bound} this leads to
\begin{align*}
\|\beta_{0}(f,g)\|_{L^{2}_{n,\tau}} &\lesssim  \Big\| \sum_{\substack{n_{1} \\n=n_{1}+n_{2}}} \int_{\tau=\tau_{1}+\tau_{2}}
\frac{|f(n_{1},\tau_{1})||g(n_{2},\tau_{2})|}{\langle \sigma_{1}\rangle^{\frac{1}{2}-\varepsilon}\langle \sigma_{2}\rangle^{\frac{1}{2}-3\varepsilon}}d\tau_{1}
\Big\|_{L^{2}_{n,\tau}} \\
&= \Big\| \sum_{\substack{n_{2} \\n=n_{1}+n_{2}}} \int_{\tau=\tau_{1}+\tau_{2}}
\frac{|f(n_{1},\tau_{1})||g(n_{2},\tau_{2})|}{\langle \sigma_{1}\rangle^{\frac{1}{2}-\varepsilon}\langle \sigma_{2}\rangle^{\frac{1}{2}-3\varepsilon}}d\tau_{2}
\Big\|_{L^{2}_{n,\tau}} \\
&\lesssim \|f\|_{L^{2}_{n,\tau}}\|g\|_{L^{2}_{n,\tau}},
\end{align*}
by the analysis done in Case 1(A).  This completes Case 1.  That is, the proof of \eqref{Eqn:bilinear2} in the region $A_{0}$ is complete.

\vspace{0.1in}
\noindent $\bullet$ \textbf{Case 2:}
$|\sigma_{1}|=\max(|\sigma_{0}|,|\sigma_{1}|,|\sigma_{2}|)$.
\newline

Using duality, we choose to establish the equivalent estimate:
\begin{align}
\|\tilde{\beta}_{1}(g,h)\|_{L^{2}_{n_{1},\tau_{1}}} \lesssim \|g\|_{L^{2}_{n,\tau}}\|h\|_{L^{2}_{n,\tau}},
\label{Eqn:bilinear3}
\end{align}
where
\[
\tilde{\beta}(g,h)(n_{1},\tau_{1}):=\sum_{\substack{n \\n=n_{1}+n_{2}}} \int_{\tau=\tau_{1}+\tau_{2}}
\frac{|n|\langle n\rangle^{s}h(n,\tau)g(n_{2},\tau_{2})}
{\langle n_{1}\rangle^{s}\langle n_{2}\rangle^{s}\langle \sigma_{0}\rangle^{\frac{1}{2}-\varepsilon}
\langle \sigma_{1}\rangle^{\frac{1}{2}-\varepsilon}\langle \sigma_{2}\rangle^{\frac{1}{2}-\varepsilon}}d\tau.
\]

\noindent From the algebraic relation $\max(|\sigma_{0}|,|\sigma_{1}|,|\sigma_{2}|) \geq |nn_{1}n_{2}|$, and by Lemma
\ref{Lemma:kernelbound}, we have
\begin{align}
\|\tilde{\beta}_{1}(g,h)\|_{L^{2}_{n_{1},\tau_{1}}}
&\leq \Big\| \sum_{\substack{n \\n=n_{1}+n_{2}}} \int_{\tau=\tau_{1}+\tau_{2}}
\frac{|N^{1}|^{4\varepsilon}|h(n,\tau)||g(n_{2},\tau_{2})|}
{\langle \sigma_{0}\rangle^{\frac{1}{2}-\varepsilon}\langle \sigma_{2}\rangle^{\frac{1}{2}-\varepsilon}}d\tau
\Big\|_{L^{2}_{n_{1},\tau_{1}}}.
\label{Eqn:case2bound}
\end{align}
Again, we handle the factor of $|N^{1}|^{4\varepsilon}$ by considering the relative sizes of $n_{1},n_{2}$.

\vspace{0.1in}
\noindent $\bullet$ \textbf{Case 2(A):}  $|n|\sim N^{1}$.
\newline

Then $\displaystyle|\sigma_{0}|\gtrsim  |N^{1}|^{2},\ \
\Rightarrow \frac{1}{\langle \sigma_{0}\rangle^{2\varepsilon}}
\lesssim \frac{1}{|N^{1}|^{4\varepsilon}}$.
Combining this estimate with \eqref{Eqn:case2bound}, and applying Cauchy-Schwarz, we find
\begin{align*}
\|\tilde{\beta}_{1}(g,h)\|_{L^{2}_{n_{1},\tau_{1}}}
&\leq \Big\| \sum_{\substack{n \\n=n_{1}+n_{2}}} \int_{\tau=\tau_{1}+\tau_{2}}
\frac{|h(n,\tau)||g(n_{2},\tau_{2})|}
{\langle \sigma_{0}\rangle^{\frac{1}{2}-3\varepsilon}
\langle \sigma_{2}\rangle^{\frac{1}{2}-\varepsilon}}d\tau
\Big\|_{L^{2}_{n_{1},\tau_{1}}}  \\
&\leq \|g\|_{L^{2}_{n,\tau}}\|h\|_{L^{2}_{n,\tau}},
\end{align*}
by the analysis done in Case 1(A).

\vspace{0.1in}
\noindent $\bullet$ \textbf{Case 2(B):}  , $|n_{1}|\sim |n_{2}|\sim N^{1}$, $|n|\sim N^{3}$.
\newline

Then $\displaystyle|\sigma_{2}|\gtrsim |N^{1}|^{2},\ \  \Rightarrow \frac{1}{\langle \sigma_{2}\rangle^{2\varepsilon}}
\lesssim  \frac{1}{|N^{1}|^{4\varepsilon}}$.  Combined with \eqref{Eqn:case2bound} this leads to
\begin{align*}
\|\tilde{\beta}_{1}(g,h)\|_{L^{2}_{n_{1},\tau_{1}}}
&\lesssim \Big\| \sum_{\substack{n \\n=n_{1}+n_{2}}} \int_{\tau=\tau_{1}+\tau_{2}}
\frac{|h(n,\tau)||g(n_{2},\tau_{2})|}
{\langle \sigma_{0}\rangle^{\frac{1}{2}-\varepsilon}
\langle \sigma_{2}\rangle^{\frac{1}{2}-3\varepsilon}}d\tau \Big\|_{L^{2}_{n_{1},\tau_{1}}}  \\
&= \Big\| \sum_{\substack{n_{2} \\n=n_{1}+n_{2}}} \int_{\tau=\tau_{1}+\tau_{2}}
\frac{|h(n,\tau)||g(n_{2},\tau_{2})|}
{\langle \sigma_{0}\rangle^{\frac{1}{2}-\varepsilon}
\langle \sigma_{2}\rangle^{\frac{1}{2}-3\varepsilon}}d\tau_{2} \Big\|_{L^{2}_{n_{1},\tau_{1}}} \\
&\lesssim \|g\|_{L^{2}_{n,\tau}}\|h\|_{L^{2}_{n,\tau}},
\end{align*}
by the analysis done in Case 1(A).

\vspace{0.1in}

\noindent $\bullet$ \textbf{Case 3:}
$|\sigma_{2}|=\max(|\sigma_{0}|,|\sigma_{1}|,|\sigma_{2}|)$.
\newline

Observe that we may rewrite (\ref{Eqn:bilinear2}) as
\begin{align}
\|\beta(f,g)\|_{L^{2}_{n,\tau}} \lesssim \|f\|_{L^{2}_{n,\tau}}\|g\|_{L^{2}_{n,\tau}},
\label{Eqn:bilinear4}
\end{align}
with
\[
\beta(f,g):=\sum_{\substack{n_{2} \\n=n_{1}+n_{2}}} \int_{\tau=\tau_{1}+\tau_{2}}
\frac{|n|\langle n\rangle^{s}f(n_{1},\tau_{1})g(n_{2},\tau_{2})}
{\langle n_{1}\rangle^{s}\langle n_{2}\rangle^{s}\langle \sigma_{0}\rangle^{\frac{1}{2}-\varepsilon}
\langle \sigma_{1}\rangle^{\frac{1}{2}-\varepsilon}\langle \sigma_{2}\rangle^{\frac{1}{2}-\varepsilon}}d\tau_{2}.
\]
Justifying (\ref{Eqn:bilinear2}) in the region $A_{2}$ is therefore equivalent to justifying (\ref{Eqn:bilinear2})
in the region $A_{1}$, and we are done by the analysis in Case 2.  This completes Case 3, and the proof of
Proposition \ref{Prop:bilinear}.

\end{proof}

\subsection{Frequency truncated global well-posedness}

In this subsection we establish Proposition \ref{Prop:finiteGWP}.

First observe that, for any fixed $N>0$, the proof of Theorem \ref{Thm:Contraction} is easily modified to produce analogous statements for the truncated system \eqref{Eqn:SDKDV-trunc}.
However, for any fixed $N>0$,
$u_{0}^{N}\in H^{s}(\mathbb{T})$ and
$\phi^{N}\in HS(L^{2},H^{s+1-2\varepsilon})$, for any $s,\varepsilon \in \mathbb{R}$,
regardless of the conditions placed on $u_{0}$, $\phi$.  That is, for fixed $N>0$, we can relax the conditions placed on $u_{0}$ and $\phi$ before applying (the finite-dimensional modification of) Theorem \ref{Thm:Contraction}.  In this subsection, we will consider $u_{0}\in H^{s}(\mathbb{T})$ and $\phi\in HS(L^{2},H^{s})$, since with $s<-\frac{1}{2}$, this condition admits $\phi=\text{Id}$, which is the long-term goal of our study.  Indeed, given $u_{0}\in H^{s}(\mathbb{T})$ and $\phi\in L^{2}(L^{2},H^{s})$, by Theorem \ref{Thm:Contraction} there is a unique solution $u^{N}(t)$ to \eqref{Eqn:SDKDV-trunc} for $t\in [0,T_{\omega,N}]$ where
\begin{align}
T_{\omega,N} = \min \Big\{T>0:2CT^{\varepsilon-}\Big(\|u^{N}_{0}\|_{H^{s}(\mathbb{T})} + 2 + \|\chi_{[0,T]}\Phi^{N}\|_{X^{s,\frac{1}{2}-\varepsilon}}\Big)^{2}
\geq 1\Big\}.
\label{Eqn:stoptime2}
\end{align}

Before we present the proof of Proposition \ref{Prop:finiteGWP} we will use \eqref{Eqn:dL2} to control the expectation of higher moments of the $L^{2}$-norm of the solution to \eqref{Eqn:SDKDV-trunc-Duh}.  In particular, we establish the following lemma.

\begin{lemma}
Suppose $\tilde{\Omega}\subset \Omega$ is such that for all $\omega\in \tilde{\Omega}$, there exists
$u^{N}(t)$, a solution to \eqref{Eqn:SDKDV-trunc-Duh} for $t\in[0,T]$, with $T\leq T_{\omega,N}$.  Then for all $t\in [0,T]$, and any integer $p\geq 0$,
\begin{align}
\mathbb{E}\Big(\sup_{0\leq t\leq T}\|u^{N}(t)\|_{L^{2}_{x}}^{2p}\cdot\chi_{\tilde{\Omega}}\Big) \leq C_{p,N},
\label{Eqn:highmombound}
\end{align}
where $C_{p,N}=C(p,T,\|u^{N}_{0}\|_{L^{2}_{x}},\|\phi^{N}\|_{H^{1}})$ is sufficiently large.
\label{Lemma:highmoments}
\end{lemma}

\begin{proof}[Proof of Lemma \ref{Lemma:highmoments}:]

\noindent  Since $T\leq T_{\omega,N}$ inside $\tilde{\Omega}$, we have
\begin{align*}
\mathbb{E}\Big(\sup_{0\leq t\leq T}\|u^{N}(t)\|_{L^{2}_{x}}^{2p}\cdot\chi_{\tilde{\Omega}}\Big) \leq \mathbb{E}\Big(\sup_{0\leq t \leq T}\|u^{N}(t\wedge T_{\omega,N})\|_{L^{2}_{x}}^{2p}\Big).
\end{align*}
To justify \eqref{Eqn:highmombound}, it therefore suffices to prove that for each integer $p\geq 0$, we have
\begin{align}
\mathbb{E}\Big(\sup_{0\leq t \leq T}\|u^{N}(t\wedge T_{\omega,N})\|_{L^{2}_{x}}^{2p}\Big) \leq C_{p,N}.
\label{Eqn:LPgrowth}
\end{align}
In this proof, we will simultaneously establish the following inequality, which we claim holds true almost surely, for every integer $p\geq 0$.
\begin{align}
\int_{0}^{T\wedge T_{\omega,N}}\|u^{N}(t)\|_{L^{2}_{x}}^{2p}\|u^{N}(t)\|^{2}_{\dot{H}^{1}_{x}}dt \leq \tilde{C}_{p,N}(\omega) < \infty.
\label{Eqn:highmomcancel}
\end{align}
The inequality \eqref{Eqn:highmomcancel} will ensure the finiteness of non-positive terms which will then be dropped from our estimates.

We proceed to prove \eqref{Eqn:LPgrowth} and \eqref{Eqn:highmomcancel} by induction on integers $p\geq 0$. For $p=0$, \eqref{Eqn:LPgrowth} is trivial.  For
\eqref{Eqn:highmomcancel} with $p=0$, we can easily modify the justification of \eqref{Eqn:drop2} in subsection 3.1, and omit the details.

Now suppose that \eqref{Eqn:LPgrowth} and \eqref{Eqn:highmomcancel} both hold up to and including some $p-1\geq 0$.  From \eqref{Eqn:dL2} and the It\^{o} formula
\[
dY^{p}=pY^{p-1}dY + \frac{1}{2}p(p-1)Y^{p-2}(dY)^{2}
\]
\noindent we find
\begin{align*}
d\Big(\|u^{N}(t)\|_{L^{2}_{x}}^{2p}\Big) &= p\|u^{N}(t)\|_{L^{2}_{x}}^{2(p-1)}\Big[\Big(-2\|u^{N}(t)\|^{2}_{\dot{H}^{1}_{x}} + 2\|\phi^{N}\|^{2}_{\dot{H}^{1}_{x}}\Big)dt \\
&\ \ \ \ \ \ - \sum_{|n|\leq N,n\neq 0}n\phi_{n}\big[-a_{n}(t)dB_{n}^{2}(t) + b_{n}(t)dB_{n}^{1}(t)\big]\Big] \\
&\ \ \ \ \ \ + p(p-1)\|u^{N}(t)\|_{L^{2}_{x}}^{2(p-2)}\sum_{|n|\leq N}
n^{2}\phi_{n}^{2}(a_{n}^{2}+b_{n}^{2})dt.
\end{align*}
For each $n$, let
$$X_{n,t}:=\int_{0}^{t}\|u^{N}(t')\|_{L^{2}_{x}}^{2(p-1)}a_{n}(t')dB_{n}^{2}(t')$$ and $$Y_{n,t}:=\int_{0}^{t}\|u^{N}(t')\|_{L^{2}_{x}}^{2(p-1)}b_{n}(t')dB_{n}^{1}(t').$$
Then we have
\begin{align*}
\|u^{N}(t)\|_{L^{2}_{x}}^{2p} - \|u^{N}_{0}\|_{L^{2}_{x}}^{2p}
&= \int_{0}^{t}p\|u^{N}(t')\|_{L^{2}_{x}}^{2(p-1)}\Big(-2\|u^{N}(t')\|^{2}_{\dot{H}^{1}_{x}} + 2\|\phi^{N}\|^{2}_{\dot{H}^{1}_{x}}\Big)dt' \notag \\
&\ + p(p-1)\int_{0}^{t}\|u^{N}(t')\|_{L^{2}_{x}}^{2(p-2)}
\sum_{|n|\leq N}
n^{2}\phi_{n}^{2}(a_{n}^{2}+b_{n}^{2})(t')dt'
\notag \\
&\ + p\sum_{|n|\leq N,n\neq 0}n\phi_{n}
\big(X_{n,t} + Y_{n,t}\big),
\end{align*}
In particular, we have
\begin{align}
\|u^{N}(t\wedge T_{\omega,N})\|_{L^{2}_{x}}^{2p} &= \|u^{N}_{0}\|_{L^{2}_{x}}^{2p} + \int_{0}^{t\wedge T_{\omega,N}}p\|u^{N}(t')\|_{L^{2}_{x}}^{2(p-1)}\Big(-2\|u^{N}(t')\|^{2}_{\dot{H}^{1}_{x}} + 2\|\phi^{N}\|^{2}_{\dot{H}^{1}_{x}}\Big)dt' \notag \\
&\ + p(p-1)\int_{0}^{t\wedge T_{\omega,N}}\|u^{N}(t')\|_{L^{2}_{x}}^{2(p-2)}
\sum_{|n|\leq N}
n^{2}\phi_{n}^{2}(a_{n}^{2}+b_{n}^{2})(t')dt'
\notag \\
&\ + p\sum_{|n|\leq N,n\neq 0}n\phi_{n}\big(X_{n,t\wedge T_{\omega,N}} + Y_{n,t\wedge T_{\omega,N}}\big).
\label{Eqn:LPcomp}
\end{align}

\noindent Using \eqref{Eqn:highmomcancel} at the level $p-1$, we have that, almost surely, for each $t\in[0,T]$,
\begin{align}
\int_{0}^{t\wedge T_{\omega,N}}\|u^{N}(t')\|_{L^{2}_{x}}^{2(p-1)}\|u^{N}(t')\|^{2}_{\dot{H}^{1}_{x}}dt'
&\leq \int_{0}^{T\wedge T_{\omega,N}}\|u^{N}(t')\|_{L^{2}_{x}}^{2(p-1)}\|u^{N}(t')\|^{2}_{\dot{H}^{1}_{x}}dt'
\notag \\
&\leq \tilde{C}_{p-1,N}(\omega) < \infty.
\label{Eqn:candrop}
\end{align}
Given \eqref{Eqn:candrop}, we can almost surely drop the second term on the right-hand side of \eqref{Eqn:LPcomp}.  That is, almost surely,
\begin{align*}
\|u^{N}(t\wedge T_{\omega,N})\|_{L^{2}_{x}}^{2p} &\leq \|u^{N}_{0}\|_{L^{2}_{x}}^{2p} + 2p\|\phi^{N}\|^{2}_{H^{1}}\int_{0}^{t\wedge T_{\omega,N}}\|u^{N}(t')\|_{L^{2}_{x}}^{2(p-1)}dt'\\
&\ \ \ \ \ \ \ + p(p-1)\int_{0}^{t\wedge T_{\omega,N}}\|u^{N}(t')\|_{L^{2}_{x}}^{2(p-2)}
\sum_{|n|\leq N}
n^{2}\phi_{n}^{2}(a_{n}^{2}+b_{n}^{2})(t')dt' \\
&\ \ \ \ \ \ \ + p\sum_{|n|\leq N}n\phi_{n}\big(X_{n,t\wedge T_{\omega,N}} + Y_{n,t\wedge T_{\omega,N}}\big)
\\
&\leq \|u^{N}_{0}\|_{L^{2}_{x}}^{2p} + 2p\|\phi^{N}\|^{2}_{H^{1}}\int_{0}^{t\wedge T_{\omega,N}}\|u^{N}(t')\|_{L^{2}_{x}}^{2(p-1)}dt' \\
&\ \ \ \ \ \ \ + p(p-1)\|\phi^{N}\|^{2}_{H^{1}}\int_{0}^{t\wedge T_{\omega,N}}\|u^{N}(t')\|_{L^{2}_{x}}^{2(p-1)}dt' \\
&\ \ \ \ \ \ \ + p\sum_{|n|\leq N}n\phi_{n}\big(X_{n,t\wedge T_{\omega,N}} + Y_{n,t\wedge T_{\omega,N}}\big)
\end{align*}
This gives
\begin{align}
\sup_{0\leq t \leq T}\|u^{N}(t\wedge T_{\omega,N})\|_{L^{2}_{x}}^{2p} &\leq \|u_{0}\|_{L^{2}_{x}}^{2p} + C(p,T) \|\phi^{N}\|_{H^{1}}^{2}\Big(\sup_{0\leq t \leq T}\|u^{N}(t\wedge T_{\omega,N})\|_{L^{2}_{x}}^{2(p-1)}\Big) \notag \\
& \ \ \ \ \ \  + \sum_{|n|\leq N,n\neq 0}|n\phi_{n}|\Big(\sup_{0\leq t \leq T}|X_{n,t\wedge T_{\omega,N}}| + \sup_{0\leq t \leq T}|Y_{n,t\wedge T_{\omega,N}}|\Big).
\label{Eqn:preburkbound}
\end{align}

\noindent  The stochastic integrals $X_{n,t\wedge T_{\omega,N}}$ and $Y_{n,t\wedge T_{\omega,N}}$ are continuous martingales.  Letting $ [X_{n}]_{t} = \int_{0}^{t}\|u^{N}(t')\|_{L^{2}_{x}}^{4(p-1)}(a_{n}(t'))^{2}dt'$
denote the quadratic variation of $X_{n,t}$, we
apply Burkholder's inequality to find, for each $n$, that
\begin{align*}
\mathbb{E}\Big(\sup_{0\leq t \leq T}|X_{n,t\wedge T_{\omega,N}}|\Big)
&\leq 3\mathbb{E}\Big(\big([X_{n}]_{T\wedge T_{\omega,N}}\big)^{\frac{1}{2}}\Big)  \\
&= 3\mathbb{E}\Big(\big(\int_{0}^{T\wedge T_{\omega,N}}\|u^{N}(t)\|_{L^{2}_{x}}^{4(p-1)}(a_{n}(t))^{2}dt\big)^{\frac{1}{2}}\Big)  \\
&\leq 3\mathbb{E}\Big(\big(\sup_{0\leq t \leq T\wedge T_{\omega,N}}\|u^{N}(t)\|_{L^{2}_{x}}^{4(p-1)}\big)^{\frac{1}{2}}\big(\int_{0}^{T\wedge T_{\omega,N}}(a_{n}(t))^{2}dt\big)^{\frac{1}{2}}\Big)  \\
&= 3\mathbb{E}\Big(\big(\sup_{0\leq t \leq T\wedge T_{\omega,N}}\|u^{N}(t)\|_{L^{2}_{x}}^{2(p-1)}\big)\big(\int_{0}^{T\wedge T_{\omega,N}}(a_{n}(t))^{2}dt\big)^{\frac{1}{2}}\Big)  \\
&\leq 3\mathbb{E}\Big(\big(\sup_{0\leq t \leq T\wedge T_{\omega,N}}\|u^{N}(t)\|_{L^{2}_{x}}^{2(p-1)}\big)
\\
&\ \ \ \ \ \ \ \ \ \ \ \ \ \big(\int_{0}^{T\wedge T_{\omega,N}}((a_{n}(t))^{2} + (b_{n}(t))^{2})dt\big)^{\frac{1}{2}}\Big)
\end{align*}
and the same inequality holds with $Y_{n,t}$
on the left-hand side.  This gives
\begin{align*}
\mathbb{E}\Bigg[&\sum_{|n|\leq N}|n\phi_{n}|\sup_{0\leq t \leq T}|X_{n,t}| \Bigg]
= \sum_{|n|\leq N}|n\phi_{n}|\mathbb{E}\Big(\sup_{0\leq t \leq T}|X_{n,t}| \Big)   \\
&\leq 3\sum_{|n|\leq N}|n\phi_{n}|\mathbb{E}\Bigg[\big(\sup_{0\leq t \leq T\wedge T_{\omega,N}}\|u^{N}(t)\|_{L^{2}_{x}}^{2(p-1)}\big) \big(\int_{0}^{T\wedge T_{\omega,N}}((a_{n}(t'))^{2} + (b_{n}(t'))^{2})dt'\big)^{\frac{1}{2}}\Bigg] \\
&=  3\mathbb{E}\Bigg[\sum_{|n|\leq N}|n\phi_{n}|\big(\sup_{0\leq t \leq T\wedge T_{\omega,N}}\|u^{N}(t)\|_{L^{2}_{x}}^{2(p-1)}\big) \big(\int_{0}^{T\wedge T_{\omega,N}}((a_{n}(t'))^{2} + (b_{n}(t'))^{2})dt'\big)^{\frac{1}{2}}\Bigg]
\\
&\leq
3\|\phi^{N}\|_{\dot{H}^{1}}\mathbb{E}\Bigg[\big(\sup_{0\leq t \leq T\wedge T_{\omega,N}}\|u^{N}(t)\|_{L^{2}_{x}}^{2(p-1)}\big)
\int_{0}^{T\wedge T_{\omega,N}}\|u^{N}(t')\|_{L^{2}_{x}}dt'\Bigg] \\ &\leq
3T\|\phi^{N}\|_{\dot{H}^{1}}\mathbb{E}\Bigg[\big(\sup_{0\leq t \leq T\wedge T_{\omega,N}}\|u^{N}(t)\|_{L^{2}_{x}}^{2(p-1)}\big)
\big(\sup_{0\leq t \leq T\wedge T_{\omega,N}}\|u^{N}(t)\|_{L^{2}_{x}}\big)\Bigg],
\end{align*}
where we have applied Cauchy-Schwarz in the second last line.  We can rearrange the supremum to obtain,
\begin{align}
\mathbb{E}\Bigg[&\sum_{|n|\leq N}|n\phi_{n}|\sup_{0\leq t \leq T}|X_{n,t}| \Bigg] \leq
3T\|\phi^{N}\|_{\dot{H}^{1}}\mathbb{E}\Big(\sup_{0\leq t \leq T\wedge T_{\omega,N}}\|u^{N}(t)\|_{L^{2}_{x}}^{2p-1}
\Big)
\notag  \\
&=
3T\|\phi^{N}\|_{\dot{H}^{1}}\mathbb{E}\Big(\sup_{0\leq t \leq T}\|u^{N}(t\wedge T_{\omega,N})\|_{L^{2}_{x}}^{2p-1}
\Big)
\notag  \\
&=
3T\|\phi^{N}\|_{\dot{H}^{1}}\mathbb{E}\Bigg[\big(\sup_{0\leq t \leq T}\|u^{N}(t\wedge T_{\omega,N})\|_{L^{2}_{x}}^{2(p-1)}\big)
^{\frac{1}{2}}
\big(\sup_{0\leq t \leq T}\|u^{N}(t\wedge T_{\omega,N})\|_{L^{2}_{x}}^{2p}\big)^{\frac{1}{2}}
\Bigg]
\notag  \\
&\leq
3T\|\phi^{N}\|_{\dot{H}^{1}}\Bigg[\mathbb{E}\Big(\sup_{0\leq t \leq T}\|u^{N}(t\wedge T_{\omega,N})\|_{L^{2}_{x}}^{2(p-1)}\Big)\Bigg]
^{\frac{1}{2}}
\Bigg[\mathbb{E}\big(\sup_{0\leq t \leq T\wedge T_{\omega,N}}\|u^{N}(t)\|_{L^{2}_{x}}^{2p}\big)
\Bigg]^{\frac{1}{2}}
\notag  \\
&\leq
C(T)\|\phi^{N}\|_{\dot{H}^{1}}^{2}\mathbb{E}\Big(\sup_{0\leq t \leq T}\|u^{N}(t\wedge T_{\omega,N})\|_{L^{2}_{x}}^{2(p-1)}\Big) + \frac{1}{2}\mathbb{E}\Big(\sup_{0\leq t \leq T}\|u^{N}(t\wedge T_{\omega,N})\|_{L^{2}_{x}}^{2p}\Big),
\label{Eqn:burkbound}
\end{align}
where we have applied Cauchy-Schwarz in the second last line.  Combining \eqref{Eqn:preburkbound} with \eqref{Eqn:burkbound} we find
\begin{align*}
\mathbb{E}\Big(\sup_{0\leq t \leq T}\|u^{N}(t\wedge T_{\omega,N})\|_{L^{2}_{x}}^{2p}\Big)
&\leq \|u_{0}^{N}\|_{L^{2}_{x}}^{2p} + C(p,T) \|\phi^{N}\|_{H^{1}}^{2}\mathbb{E}\Big(\sup_{0\leq t \leq T}\|u^{N}(t\wedge T_{\omega,N})\|_{L^{2}_{x}}^{2(p-1)}\Big)  \\ &\
\ \ \ \ + \frac{1}{2}\mathbb{E}\Big(\sup_{0\leq t \leq T}\|u^{N}(t\wedge T_{\omega,N})\|_{L^{2}_{x}}^{2p}\Big).
\end{align*}
Upon rearrangement, we conclude that
\begin{align*}
\mathbb{E}\Big(\sup_{0\leq t \leq T}\|u^{N}(t\wedge T_{\omega,N})\|_{L^{2}_{x}}^{2p}\Big)
&\leq 2\|u_{0}^{N}\|_{L^{2}_{x}}^{2p} + C(p,T) \|\phi^{N}\|_{H^{1}}^{2}\mathbb{E}\Big(\sup_{0\leq t \leq T}\|u^{N}(t\wedge T_{\omega,N})\|_{L^{2}_{x}}^{2(p-1)}\Big).
\end{align*}
Applying our hypothesis, that \eqref{Eqn:LPgrowth} holds up to $p-1$,
\begin{align*}
\mathbb{E}\Big(\sup_{0\leq t \leq T}\|u^{N}(t\wedge T_{\omega,N})\|_{L^{2}_{x}}^{2p}\Big)
&\leq
2\|u_{0}^{N}\|_{L^{2}_{x}}^{2} + C(p,T) \|\phi^{N}\|_{H^{1}}^{2}C_{p-1,N} \\
&\leq C_{p,N},
\end{align*}
by taking $C_{p,N}=C(p,T,\|u^{N}_{0}\|_{L^{2}_{x}},\|\phi^{N}\|_{H^{1}})$
sufficiently large.  Thus \eqref{Eqn:LPgrowth} holds at the level $p$.  In particular, since $p\geq 1$, we know that \eqref{Eqn:LPgrowth} holds at the level $p=1$.

It remains to establish \eqref{Eqn:highmomcancel} at the level $p$.  We find
\begin{align*}
\int_{0}^{T\wedge T_{\omega_{K}}}\|u^{N}(t)\|_{L^{2}_{x}}^{2p}\|u^{N}(t)\|^{2}_{\dot{H}^{1}_{x}}dt  &\leq  \Big(\sup_{0\leq t\leq T\wedge T_{\omega,N}}\|u^{N}(t)\|_{L^{2}_{x}}^{2}\Big)\int_{0}^{T\wedge T_{\omega,N}}\|u^{N}(t)\|_{L^{2}_{x}}^{2(p-1)}\|u^{N}(t)\|^{2}_{\dot{H}^{1}_{x}}dt  \\
&=  \Big(\sup_{0\leq t\leq T}\|u^{N}(t\wedge T_{\omega,N})\|_{L^{2}_{x}}^{2}\Big)\int_{0}^{T\wedge T_{\omega,N}}\|u^{N}(t)\|_{L^{2}_{x}}^{2(p-1)}\|u^{N}(t)\|^{2}_{\dot{H}^{1}_{x}}dt  \\
&\leq C(\omega) < \infty,
\end{align*}

\noindent almost surely, by \eqref{Eqn:highmomcancel} at the level $p-1$, and \eqref{Eqn:LPgrowth} holds at the level $p=1$.  Thus \eqref{Eqn:LPgrowth} and \eqref{Eqn:highmomcancel} both hold for each $p\geq 0$.  In particular, \eqref{Eqn:highmombound} holds for each $p\geq 0$ and the proof of Lemma \ref{Lemma:highmoments} is complete.

\end{proof}

\begin{proof}[Proof of Proposition \ref{Prop:finiteGWP}]
Let $T>0$.  We first claim that for all $\sigma >0$,
$\exists \,\Omega_{\sigma}\subset \Omega$ such that
$P(\Omega_{\sigma}^{c})<\sigma$ and for all $\omega \in \Omega_{\sigma}$ a unique solution of \eqref{Eqn:SDKDV-trunc} exists for $t\in [0,T]$.

For any $0<\delta <T$, we split $[0,T]$ into $M\sim \frac{T}{\delta}$ subintervals.
Let
$$ \Omega_{0} = \cap_{k=1}^{M}\Big\{\|\chi_{[(k-1)\delta,\delta]}
\int_{(k-1)\delta}^{t}S(t-t')\phi^{N} \partial_{x} dW(t')\|
_{X^{0,\frac{1}{2}-\varepsilon}} \leq L
\Big\}. $$
By Chebyshev and Proposition \ref{Prop:EstofStochConv} we have
\begin{align*}
P\big(\Omega_{0}^{c}\big) &\leq \sum_{k=1}^{M}
P\Big(\|\chi_{[(k-1)\delta,k\delta]}
\int_{(k-1)\delta}^{t}S(t-t')\phi^{N} \partial_{x} dW(t')\|
_{X^{0,\frac{1}{2}-\varepsilon}} > L
\Big)  \\
&\leq
\sum_{k=1}^{M}
\frac{\mathbb{E}\Big(\|\chi_{[(k-1)\delta,k\delta]}
\int_{(k-1)\delta}^{t}S(t-t')\phi^{N} \partial_{x} dW(t')\|^{2}
_{X^{0,\frac{1}{2}-\varepsilon}}
\Big)}{L^{2}}  \\
&\sim
\sum_{k=1}^{M}
\frac{\delta\|\phi^{N}\|_{H^{1-2\varepsilon}}^{2}}{L^{2}}  \\
&\leq
\frac{M \delta\|\phi^{N}\|_{H^{1-2\varepsilon}}^{2}}{L^{2}}  \\
&\lesssim
\frac{T\|\phi^{N}\|_{H^{1-2\varepsilon}}^{2}}{L^{2}}
< \frac{\sigma}{2},
\end{align*}
for $L(\sigma,T,\|\phi^{N}\|_{H^{1-2\varepsilon}})$ sufficiently large, independent of $\delta$.  Let $B,R>0$ be positive constants to be determined later on.  Depending on our choice of $R$, we will select $B=B(R,\|u_{0}^{N}\|_{L^{2}})$ sufficiently large such that $BR \geq \|u_{0}^{N}\|_{L^{2}}$.  Then if
$\omega \in \Omega_{0}$, we have
$$ \|u_{0}^{N}\|_{L^{2}} \leq BR \ \ \text{and}\ \
\|\chi_{[0,\delta]}
\int_{0}^{t}S(t-t')\phi^{N} \partial_{x} dW(t')\|
_{X^{0,\frac{1}{2}-\varepsilon}} \leq L. $$
By Theorem \ref{Thm:Contraction} there exists a unique solution $u^{N}(t)$ to \eqref{Eqn:SDKDV-trunc} for $t\in[0,\delta]$, where we can take
\begin{align}
T_{\omega}\geq \delta \geq \frac{T_{\omega}}{2}\gtrsim \frac{1}{(BR+L)^{\frac{2}{\varepsilon}+}}.
\label{Eqn:deltasize}
\end{align}
Now for each $t\in[0,\delta]$, $f(t):=\chi_{\Omega_{0}}\cdot\|u^{N}(t)\|_{L^{2}_{x}}$ is $\mathcal{F}_{T}$-measurable.  In particular, the set $$\Omega_{1}=\{\omega\in \Omega_{0}:
\|u^{N}(\delta)\|_{L^{2}_{x}}\leq BR\} = \Omega_{0}\cap \{f(t)\leq BR\} \in \mathcal{F}_{T}.$$
Moreover, for $\omega \in \Omega_{1}$, the same arguments extend the solution $u^{N}(t)$ to $[\delta,2\delta]$.  Repeating this procedure, on the set
$$\Omega_{n}=\{\omega\in \Omega_{n-1}:
\|u^{N}(n\delta)\|_{L^{2}_{x}}\leq BR\}\in \mathcal{F}_{T},$$
the solution $u^{N}(t)$ exists for $t\in [0,n\delta]$.
Taking $m\sim \frac{T}{\delta}$, for each $\omega\in\Omega_{m}$, the solution $u^{N}(t)$ to \eqref{Eqn:SDKDV-trunc} exists for $t\in [0,T]$.

Next observe that, for each $1\leq k\leq m$,
$k\delta-(k-1)\delta =\delta \leq T_{\omega,N}$, and we can apply Lemma \ref{Lemma:highmoments} on the time interval $[(k-1)\delta,k\delta]$, with $\tilde{\Omega}=\Omega_{k}$, to conclude that
\begin{align}
\mathbb{E}\Big(\sup_{(k-1)\delta \leq t\leq k\delta }\|u^{N}(t)\|_{L^{2}_{x}}^{2p}\cdot
\chi_{\Omega_{k}}\Big) \leq C_{p,N},
\label{Eqn:highmombound-applied}
\end{align}
where $C_{p,N}=C(T,\|u^{N}_{0}\|_{L^{2}},\|\phi^{N}\|_{H^{1}})$ is sufficiently large.  Then by Chebyshev and \eqref{Eqn:highmombound-applied}, we have that for each $1\leq k\leq m$,
\begin{align}
P\Big(\Omega_{k-1}\cap \Omega_{k}^{c}\Big) &= P\Big(\Omega_{k-1}\cap \{\|u(k\delta)\|_{L^{2}_{x}}>BR\}\Big) \notag \\ &\leq
\frac{\mathbb{E}\Big(\|u(k\delta)\|^{2p}_{L^{2}_{x}}\cdot \chi_{\Omega_{k-1}}\Big)}{(BR)^{p}} \notag \\
&\leq
\frac{C_{p,N}}{(BR)^{2p}} \notag \\
&\leq \frac{1}{B^{2p}},
\label{Eqn:Bbound}
\end{align}
\noindent by taking $R=(C^{p,N})^{\frac{1}{2p}}$, where $p$ has yet to be selected.

Now we find, from the nesting $\Omega_{0} \supset
\Omega_{1}\supset \cdots \Omega_{m}$, and \eqref{Eqn:Bbound}, that
\begin{align}
P(\Omega_{m}^{c}) &\leq  P(\Omega_{0}^{c}) + \sum_{k=1}^{m}P(\Omega_{k-1}\cap \Omega_{k}^{c}) \notag \\
&\leq  \frac{\sigma}{2} + \sum_{k=1}^{m}\frac{1}{B^{2p}} \notag \\
&=  \frac{\sigma}{2} + \frac{m}{B^{2p}}  \notag \\
&\leq  \frac{\sigma}{2} + 2\frac{T}{\delta}\frac{1}{B^{2p}} \notag  \\
&=  \frac{\sigma}{2} + 2\frac{T}{C}\frac{(BR+L)^{\frac{1}{2\varepsilon}+}}{B^{2p}} \ \ \ \text{by}\ \ \eqref{Eqn:deltasize}, \notag \\
&\leq  \sigma,
\label{Eqn:msize}
\end{align}
by selecting $p=p(\varepsilon)$ sufficiently large, and $B=B(R,L,T,\sigma)$ sufficiently large.

Given $T>0$, $\sigma >0$, let $\Omega_{\sigma}:= \Omega_{m}$, with $m$ as above.  With \eqref{Eqn:msize} and the arguments above, we have established the existence of set $\Omega_{\sigma}\in \mathcal{F}_{T}$ with $P(\Omega_{\sigma}^{c})<\sigma$, such that for all $\omega \in \Omega_{\sigma}$, there is a unique solution $u^{N}(t)$ to \eqref{Eqn:SDKDV-trunc} on $[0,T]$.

Taking $\rho=\cup_{n=1}^{\infty}\Omega_{\frac{1}{n}}$, we have $P(\rho^{c})\leq P(\Omega_{\frac{1}{n}}^{c}) < \frac{1}{n}$ for every $n$, and therefore $P(\rho)=1$.  Furthermore, if $\omega\in \rho$,
then $\omega\in \Omega_{\frac{1}{n}}$ for some $n\in \mathbb{N}$, and there is a unique solution $u^{N}(t)$ to \eqref{Eqn:SDKDV-trunc} on $[0,T]$.
The proof of Proposition \ref{Prop:finiteGWP} is complete.

\end{proof}

\end{document}